\documentclass[12pt,a4paper]{article} 
\usepackage{amsmath,amssymb,amsthm,amsfonts,amscd,euscript,verbatim, t1enc, newlfont}
\usepackage{hyperref}
\usepackage{graphicx}
\usepackage[all]{xy}
 \theoremstyle{plain}
\newtheorem{theo}{Theorem}[subsection]

\newtheorem{pr}[theo]{Proposition}

 \newtheorem{lem}[theo]{Lemma}

\theoremstyle{remark}
\newtheorem{rema}[theo]{Remark}

\theoremstyle{definition}
\newtheorem{defi}[theo]{Definition}

\newcommand\cu{\underline{C}}
\newcommand\du{\underline{D}}
\newcommand\md{\mathcal{D}}
\newcommand\eu{\underline{E}}

\newcommand\gu{\underline{G}}

\newcommand\q{{\mathbb{Q}}}
\newcommand\obj{\operatorname{Obj}}
\newcommand\hw{{\underline{Hw}}}

\newcommand\ab{{Ab}}
\DeclareMathOperator\cha{\operatorname{char}}

\newcommand\zop{{\mathbb{Z}[\frac{1}{p}]}}
\newcommand\n{{\mathbb{N}}}
\newcommand\z{{\mathbb{Z}}}

\newcommand\wchow{w_{Chow}{}}

\newcommand\ns{\{0\}}
\newcommand\al{\alpha}
\newcommand\be{\beta}

\newcommand\de{\delta}

\newcommand\lam{\Lambda}

\newcommand\mgq{\mathcal{M}_\q{}}
\newcommand\mgbm{\mathcal{M}^{BM}} 
\newcommand\mgbms{\mathcal{M}^{BM}_S}
\newcommand\mgbmsz{\mathcal{M}^{BM}_{B}}
\newcommand\mgbmx{\mathcal{M}^{BM}_X}

\newcommand\modd{\operatorname{Mod}}
\newcommand\spe{\operatorname{Spec}}
\newcommand\codim{\operatorname{codim}}

\newcommand\p{\mathbb{P}}
\newcommand\id{\operatorname{id}}
 \newcommand\lan{\langle}
\newcommand\ra{\rangle}
\newcommand\bl{\bigl(} \newcommand\br{\bigl)}

\newcommand\spv{\operatorname{SmPrVar}}

\DeclareMathOperator\kar{\operatorname{Kar}}
 \DeclareMathOperator\ke{\operatorname{Ker}}

\DeclareMathOperator\inli{\varinjlim}
\newcommand\ob{^{-1}}
\newcommand\ihom{{\underline{Hom}}}
\newcommand\hu{\underline{H}}

\newcommand\dmer{DM^{eff}_{R}}
\newcommand\dmers{DM^{eff}_{-,R}}
\newcommand\dmr{DM_{R}}

\newcommand\mgr{\mathcal{M}_R} 

\newcommand\chower{\underline{Chow}^{eff}_R}

\newcommand\thom{t_{hom}}

\newcommand\dmger{{DM^{\,\, {eff}}_{\scalebox{0.7}{gm,R}}}{}}

\newcommand\dmgep
{{DM^{\,\, {eff}}_{\scalebox{0.6}{gm},\zop}}{}}

\newcommand\dmgepr 
{{DM^{\,\, {eff}}_{\scalebox{0.5}{gm},\zop,R}}{}}

\newcommand\dmgepq 
{{DM^{\,\, {eff}}_{\scalebox{0.5}{gm},\zop,\q}}{}}

\newcommand\sz{B}
\newcommand\bb{{\mathcal{B}}}
\newcommand\sss{{\mathcal{S}}}
\newcommand\xx{{\mathcal{X}}}
\newcommand\oo{{\pmb{1}}}

\newcommand\dmgeq{DM^{eff}_{gm,\q}}

\newcommand\sht{SH}

\newcommand\she{SH^{eff}}
\newcommand\mgl{\operatorname{MGl}}
\newcommand\mglmod{\operatorname{MGl}-\modd}
\newcommand\shmgl{D^{\operatorname{MGl}}}

\newcommand\cp{\mathcal{P}}

\numberwithin{equation}{subsection}
\begin{document}

\title
 {Intersecting the dimension  and slice filtrations for (relative) motivic categories}
 \author{Mikhail V.Bondarko\thanks{The  work is supported by Russian Science Foundation grant no. 16-11-10200. 
}}
 \maketitle
\begin{abstract}
In this paper we prove that the intersections of the levels of the dimension filtration on Voevodsky's motivic complexes over a field $k$ with the levels of the slice one are "as small as possible", 
 i.e., that $\obj d_{\le m}\dmers\cap \obj \dmers(i)=\obj d_{\le m-i}\dmers(i)$ (for $m,i\ge 0$ and $R$ being any coefficient ring in which the exponential characteristic of 
 $k$ invertible). This statement is applied to prove that a conjecture of J. Ayoub is equivalent to a certain orthogonality assumption.
We also establish a vast generalization of our intersection result to relative motivic categories (that are required to fulfil a certain list of "axioms"). In the process we prove several new properties of relative motives and of the so-called Chow weight structures for them, and define a new modification of Gabber's dimension functions. 
\end{abstract}

\tableofcontents

\section*{Introduction}
The slice and the dimension filtrations for (various versions of) 
Voevodsky motives are quite well-known (and easy to define); yet our understanding of motives of dimension at most $m$ is quite limited already for $m=2$. In the current paper we compute the intersection of $\obj d_{\le m}\dmers$ with $\obj\dmers(i)$; here  $m,i\ge 0$ and $R$ is any coefficient ring in which the  characteristic of the base field is
 invertible if it is positive. Note that this intersection certainly contains  $\obj d_{\le m-i}\dmers(i)$ (this class is zero if $m< i$), and we prove that this inclusion is actually an equality; this result is completely new. We also prove a vast "relative motivic" generalization of this intersection statement. We use our result (for motives with rational coefficients over a field) to prove that 
Conjecture 4.22 of \cite{ayconj} is equivalent to several other assumptions.

Now we describe our results in more detail. Our first object of study is the category $\dmer$ of (unbounded) $R$-linear motivic complexes over a perfect field $k$ (for $R$ as above). 
The category $d_{\le m}\dmer$ is the localizing subcategory of $\dmer$ generated by the motives of (smooth projective) varieties of dimension at most $m$. For any $i\ge 0$ this subcategory certainly contains $d_{\le m-i}\dmer(i)=d_{\le m-i}\dmer\lan i \ra$; here $\lan i \ra$ denotes the tensor product by the $i$th power of the Lefschetz motif $R\lan 1\ra$ and we set $d_{\le j}\dmer=\ns$ if $j<0$. The question is whether the intersection of $\obj d_{\le m}\dmer$ with $\obj \dmer(i)$ (the $i$th level of the slice filtration for $\dmer$) equals $\obj d_{\le m-i}\dmer(i)$.

The starting point of our arguments is that these categories are endowed with so-called Chow weight structures (that are 
 closely related to Chow weight structures introduced in \cite{bws} and \cite{bzp}). 
Combing this observation with the results of \cite{bkw} we prove that all {\it Chow-bounded below} elements (i.e., those whose "weights" are bounded from below; here we use the homological convention for the "numeration of weights") of $\obj d_{\le m}
\allowbreak \dmer\cap \obj \dmer(i)$ belong $\obj d_{\le m-i}\dmer(i)$. Moreover, all elements of  $\obj d_{\le m}\dmer\cap \obj \dmer(i)$ become {\it right Chow-weight degenerate} in the Verdier quotient $\dmer/d_{\le m-i}\dmer(i)$ (i.e., "their weights are infinitely small" in this localization). We use the latter statement to prove that $\obj d_{\le m}\dmers\cap \obj \dmers(i)=\obj d_{\le m-i}\dmers(i)$; here $\dmers$ is the $R$-linear version of the category of bounded above motivic complexes (so, this is the category that was originally considered by Voevodsky in \cite{1} and \cite[\S14]{vbook}, 
whereas the whole $\dmer$ was only introduced in  \cite{bev} and \cite{degmod}).

Next we recall that in \cite{ayconj} J. Ayoub has introduced several interesting conjectures relating the "slice functors" with the dimension filtration (for $R$ being a $\q$-algebra).
In particular, his Conjecture 4.22 states that the functor $\ihom_{\dmer}(R\lan 1\ra,-) $ 
sends $\obj d_{\le m}\dmer$ into $\obj d_{\le m-1}\dmer$
for any $m\ge 0$. Our results easily yield that this statement is fulfilled whenever the right adjoint $\nu^{\ge 1}\cong \ihom_{\dmer}(R\lan 1\ra,-) \lan 1\ra$ to the embedding $\dmer(1)\to \dmer$ 
sends $\obj d_{\le m}\dmer$ into itself (see Proposition \ref{rayouconjred} for a more general formulation). We also prove that both of these conjectures are equivalent to the non-existence of non-zero morphisms from $\dmger(1)$ into  $d_{\le m}\dmger$ in the localization $\dmger/d_{\le m-1}\dmger$ (so, in this assumption it suffices to consider compact motives only). 

We also establish a vast "relative motivic" generalization of the aforementioned intersection calculation. So, we consider a {\it motivic triangulated} functor $\md$ from the category of essentially finite type $\sz$-schemes into the $2$-category of (compactly generated) triangulated categories, where $\sz$ is a 
 Noetherian separated excellent scheme of finite Krull dimension. Following Definition 2.2.1 of \cite{bondegl}
we define a certain "clever" slice filtration on $\md(\sz)$ (in the terms of the corresponding {\it Borel-Moore objects} and dimension functions; this definition is also related to \cite[\S2]{pelaez}). We also define a certain dimension filtration on $\md(\sz)$ using somewhat similar methods. Under certain restrictions on $\md$ these subcategories of $\md(\sz)$ are endowed with certain Chow weight structures; these weight structures are compatible with the Chow weight structure on the whole $\md(\sz)$ (whose particular cases were considered in \cite{bonivan}, \cite{brelmot}, and \cite{hebpo}). Under the additional assumption of {\it homotopy compatibility} of $\md$
(defined 
 following \cite[\S3.2]{bondegl}) we are able to prove a statement on intersections of the levels of our filtrations that precisely generalizes the aforementioned result over a field. These results can be applied to two distinct versions of relative Voevodsky motives (to the Beilinson motives introduced in \cite{cd} and to the $cdh$-motives of \cite{cdint}) as well as to certain categories of $\mgl$-modules.

Now we describe the original motivation of the author to study these filtration questions. It comes from the paper \cite{bscwh} that was  dedicated to 
 various criteria ensuring that an object $M$ of $ \dmger$ belongs to a given level of a certain filtration on motives (including dimension, slice, weight, and connectivity filtrations).\footnote{The 
criteria were formulated in terms of the new {\it Chow-weight} homology and cohomology theories. These statements vastly generalize the well-known {\it decomposition of the diagonal} results.} This  motivated the author to study the "interaction" between these filtrations, and it turned out that weight structures yield convenient general methods for 
questions of this sort.

Let us  now describe the contents  of the paper. Some more information of this sort can be found at the beginnings of sections.

In \S\ref{sews} we recall some basics on triangulated categories and weight structures on them (along with introducing some notation). This section contains just a few new results, and those   are closely related to \cite{bkw}.

 In \S\ref{sayoconj} we prove the aforementioned statements on motives over a field.

In \S\ref{srelmot} we study motivic triangulated categories satisfying certain conditions. Similarly to \cite{bondegl}, one of our main tools is a certain (co)niveau spectral sequence for the cohomology of Borel-Moore objects. It allows us to reduce quite  interesting vanishing statements to the corresponding vanishing of "motivic cohomology" over fields. We obtain that our subcategories of $\md(\sz)$ are endowed with certain Chow weight structures whenever $\md$ satisfies the Chow-compatibility condition of \S\ref{saddort} (along with a long list of "structural" properties). Under the assumption that $\md$ is also homotopy-compatible we obtain that the intersection result of  \S\ref{sayoconj} carries over to this relative context. Moreover, in this section we introduce certain new types of {\it dimension functions} essentially generalizing Gabber's ones.

The author is deeply grateful to prof. J. Ayoub for an interesting discussion of his conjectures, and to the referee for his helpful remarks.

\section{Weight structures 
and "compactly purely generated" intersections}\label{sews}


This section is mostly  dedicated to recollections; still most of the results of \S\ref{sloc} are new.

In \S\ref{snotata} we introduce some notation and conventions for (mostly, triangulated) categories; we also recall some basics on compactly generated categories.


In \S\ref{ssws} we recall  some basics  
 on weight structures.

In \S\ref{sloc} we 
relate weight structures to localizations and prove a new result on "intersections of purely compactly generated subcategories"  (this is the basic abstract result of this paper).


\subsection{Notation and basics (on compactly generated categories)}\label{snotata}

Assume that $C$ is an additive category and $X,Y\in\obj C$. 

\begin{itemize}
\item For a category $C$ and  $X,Y\in\obj C$ we will write $C(X,Y)$ for the set of morphisms from $X$ into $Y$ in $C$.

\item For a category $C'$ we will write $C'\subset C$ if $C'$ is a full 
subcategory of $C$.

\item 
We will say that $X$ is a {\it
retract} of $Y$ 
 if $\id_X$ can be 
 factored through $Y$.\footnote{If $C$ is triangulated then $X$ is a retract of $Y$ if and only if $X$ is its direct summand.}

\item Let $\hu$ be an additive subcategory  of  $C$. 
Then $\hu$  is said to be  {\it Karoubi-closed}
  in $C$ if it contains all retracts of its objects in $C$.
The full subcategory $\kar_{C}(\hu)$ of  $C$ (here "$\kar$" is for Karoubi) whose objects
are all $C$-retracts of objects  $\hu$  will be
called the {\it Karoubi-closure} of $\hu$ in $C$. 

\item The {\it  Karoubi envelope} $\kar(\hu)$ (no lower index) of an additive
category $\hu$ is the category of ``formal images'' of idempotents in $\hu$.
So, its objects are the pairs $(A,p)$ for $A\in \obj \hu,\ p\in \hu(A,A),\ p^2=p$, and the morphisms are given by the formula 
$$\kar(\hu)((X,p),(X',p'))=\{f\in B(X,X'):\ p'\circ f=f \circ p=f \}.$$ 
 
The correspondence  $A\mapsto (A,\id_A)$ (for $A\in \obj \hu$) fully embeds $\hu$ into $\kar(\hu)$.
 Besides, $\kar(\hu)$ is {\it  Karoubian}, i.e.,  any idempotent morphism in it yields a direct sum decomposition. 

\item $\cu$ below will always denote some triangulated category;
usually it will
be endowed with a weight structure $w$. 

\item For any  $A,B,C \in \obj\cu$ we will say that $C$ is an {\it extension} of $B$ by $A$ if there exists a distinguished triangle $A \to C \to B \to A[1]$.

\item A class $B\subset \obj \cu$ is said to be  {\it extension-closed}
    if it 
		is closed with respect to extensions and contains $0$. 
		The smallest  extension-closed subclass of $\obj
\cu$ containing a given $D\subset \obj \cu$ is called the {\it extension-closure} of $D$.

\item We will call the smallest Karoubi-closed  extension-closed subclass of $\obj \cu$ containing a given $D\subset \obj \cu$ the {\it envelope} of $D$.

\item Given a class $D$ of objects of $\cu$ we will write $\lan D\ra$ for the smallest full Karoubi-closed
triangulated subcategory of $\cu$ containing $D$. We will also call  $\lan D\ra$  the triangulated category {\it densely generated} by $D$.

\item For $X,Y\in \obj \cu$ we will write $X\perp Y$ if $\cu(X,Y)=\ns$. For
$D,E\subset \obj \cu$ we write $D\perp E$ if $X\perp Y$ for all $X\in D,\
Y\in E$; sometimes we will also say that $D$ is (left) {\it orthogonal} to $E$.

\item Given $D\subset\obj \cu$ we  will write $D^\perp$ for the class
$$\{Y\in \obj \cu:\ X\perp Y\ \forall X\in D\}.$$
Dually, ${}^\perp{}D$ is the class
$\{Y\in \obj \cu:\ Y\perp X\ \forall X\in D\}$.
\end{itemize}


We will  need the following simple properties of Verdier localizations. 

\begin{lem}\label{ladj}
Let $\eu\subset \du\subset \cu$ be triangulated categories, $M\in \obj \cu$.  
Assume that the Verdier quotient $\cu'=\cu/\eu$ exists; denote the localization functor $\cu\to \cu/\eu$ by $\pi$ (and recall that it is identical on objects). 
Then the following statements are valid.

1. The localization $\du'=\du/\eu$ exists also, and the restriction of $\pi$ to $\du$ gives a full  embedding of $\du/\eu$ into $ \cu'$.

2.  If $\obj \eu\perp M$ then for any $N\in \obj \cu$ the 
map $\cu(N,M)\to \cu'(\pi(N),\pi(M))$ is a bijection. 

3. Assume 
that a right adjoint $F$ to the embedding $\du\to \cu$ exists. Then $F$ also yields a (well-defined) functor   $\cu'\to \du'$ that is right adjoint to the embedding $\du'\to \cu'$ given by  assertion 1. In particular, 
we have $\pi(F(M)) =0$ 
 if and only if $\pi(\obj \du)\perp\pi (M)$. 

\end{lem}
\begin{proof}
1,2. Obvious from the description of morphisms in Verdier localizations (see \S2.1 of \cite{neebook} for the latter).

3. The first part of the assertion is given by Lemma 9.1.5 of ibid.; to obtain its second part one should just apply the definition of a right adjoint functor.  
\end{proof}

Now let us assume till the  end of this subsection that $\cu$ is a triangulated category closed with respect to (small) 
coproducts.\footnote{Recall that any triangulated category satisfying this condition is Karoubian; see Proposition 1.6.8 of \cite{neebook}.} 
 
We recall a few notions related to this setting.

We will call a 
subcategory
$\du\subset \cu$ the {\it localizing subcategory generated by} some $D\subset \obj\du$ (or say that $D$ generates $\du$ as a localizing subcategory of $\cu$) if $\du$ is the smallest full strict triangulated subcategory of $\cu$  that is closed with respect to  coproducts and contains $D$. 

Moreover, we will call the smallest strict 
subclass of $\obj \cu$ that it closed with respect to (small) coproducts, extensions, the shift $[1]$, 
 and contains $D$ the {\it pre-aisle generated by $D$} (this terminology was essentially introduced in \cite{talosa}). 

An object $M$ of $\cu$ is said to be {\it compact} if the functor $\cu(M,-)$ commutes with all small coproducts.
We will say that $\cu$ is {\it compactly generated}  if
its full (triangulated) subcategory $\cu^c$ of compact objects is essentially small and $\obj \cu^c$ generates $\cu$ as its own localizing subcategory. 

For a compactly generated $\cu$ and $\hu\subset \cu^c$  will say that  $\cu$ is {\it compactly generated by $\hu$}  if $\obj \hu$ generates $\cu$ as its own localizing subcategory (also). Recall that the latter condition is fulfilled if and only if $\cu^c$ is densely generated by $\obj\hu$; see Lemma 4.4.5 of \cite{neebook}.


The following well-known lemma will be applied several times throughout the paper.

\begin{lem}\label{lwcg} 
Let $D$ be a set of compact objects of $\cu$.

Then for $\du$ being the localizing subcategory generated by $D$ the following statements are valid.

1. The Verdier quotient category $\cu'=\cu/\du$ exists (i.e., its hom-classes are sets); it is closed with respect to coproducts.

2. The localization functor $\pi:\cu\to \cu'$ respects  coproducts and converts compact objects into compact ones. 


3. The restriction of $\pi$ to the triangulated subcategory $\cu^c\subset \cu$ of compact objects of $\cu$ gives a full embedding of $\cu^c/\lan D \ra$ into $\cu'$. 

4. If some class $C\subset \obj \cu$  generates $\cu$  as its own localizing  subcategory then  
$\pi(C)$ generates $\cu'$ as its own localizing  subcategory. 
\end{lem}
\begin{proof}
Assertions 1--3 easily follow 
from the results of \cite{neebook} (cf.  Proposition 4.3.1.3(III.1--2) of \cite{bososn} for a  closely related statement). 

 Indeed, Theorem 8.3.3 of \cite{neebook} implies that $\du$ 
   satisfies the Brown  representability condition (see Definition 8.2.1 of ibid.). Hence  Proposition 9.1.19 of ibid. yields the existence of $\cu'$. 
Moreover,   $\pi$ respects coproducts according to Corollary 3.2.11 of  ibid. The restriction of $\pi$ to $\cu^c$ is a full embedding according to Corollary 4.4.2 of ibid.

To finish the proof of assertion 2 it remains to verify that $\pi(M)$ is compact in $\cu'$ for any 
compact object $M$ of $\cu$. We fix $M$; denote by $\du'$ the localizing subcategory of $\cu$ generated by 
$D\cup \{M\}$. Then the embedding $\du'\to \cu$ (also) possesses a right adjoint $F$ according to 
 Theorem 8.4.4 of ibid. Obviously, $\obj \du'^{\perp}$ is closed with respect to $\cu$-coproducts (cf. Proposition 1.2.6(III) of \cite{bpgws}); hence $F$ respects coproducts according to Proposition 1.3.4(4) of 
 ibid. Next,  there is an  embedding functor $i:\du'/\du\to \cu'$  that possesses a right adjoint $F'$ according to Lemma \ref{ladj}(1,3). $F'$ respects coproducts also (here we invoke Corollary 3.2.11 of \cite{neebook}  once again). Thus $i$ respects the compactness of objects (obvious from the adjunction of $i$ with $F'$). Lastly, the localization $l: \du'\to \du'/\du$ respects the compactness of objects according to Theorem 4.4.9 of ibid.; thus $\pi(M)=i\circ l(M)$ is compact indeed.

Assertions 4 is obvious since $\pi$ respects coproducts (see assertion 2).
\end{proof}

\subsection{Weight structures: reminder}\label{ssws}

\begin{defi}\label{dwstr}

I. A pair of subclasses $\cu_{w\le 0},\cu_{w\ge 0}\subset\obj \cu$ 
will be said to define a weight
structure $w$ for a triangulated category  $\cu$ if 
they  satisfy the following conditions.

(i) $\cu_{w\le 0}$ and $ \cu_{w\le 0}$ are 
Karoubi-closed in $\cu$. 

(ii) {\bf Semi-invariance with respect to translations.}

$\cu_{w\le 0}\subset \cu_{w\le 0}[1]$, $\cu_{w\ge 0}[1]\subset
\cu_{w\ge 0}$.

(iii) {\bf Orthogonality.}

$\cu_{w\le 0}\perp \cu_{w\ge 0}[1]$.

(iv) {\bf Weight decompositions}.

 For any $M\in\obj \cu$ there
exists a distinguished triangle
\begin{equation}\label{wd}
X\to M\to Y
{\to} X[1]
\end{equation} 
such that $X\in \cu_{w\le 0},\  Y\in \cu_{w\ge 0}[1]$.

\end{defi}

We will also need the following definitions.

\begin{defi}\label{dwso}

Let $i,j\in \z$.

\begin{enumerate}
\item\label{id1} The full subcategory  $\hw\subset \cu$ whose object class is $\cu_{w=0}=\cu_{w\ge 0}\cap \cu_{w\le 0}$ 
 is called the {\it heart} of  $w$.

\item\label{id2} $\cu_{w\ge i}$ (resp. $\cu_{w\le i}$,  $\cu_{w= i}$) will denote the class $\cu_{w\ge 0}[i]$ (resp. $\cu_{w\le 0}[i]$,  $\cu_{w= 0}[i]$).

We  will call $\cup_{i\in \z}\cu_{w\ge i}$ the class of $w$-{\it bounded below} objects. 

\item\label{id3} The class $\cu_{w\ge i}\cap \cu_{w\le j}$ will be denoted by $\cu_{[i,j]}$  (so, it equals $\ns$ if $i>j$).

$\cu^b\subset \cu$ will be the category whose object class  is the class $\cup_{i,j\in \z}\cu_{[i,j]}$ of {\it $w$-bounded objects}.

 
\item\label{ideg} 
 We will call  elements of $\cap_{i\in \z}\cu_{w\le i}$ 
{\it  right $w$-degenerate} ones. 

$w$ will be called 
{\it  left non-degenerate} if  $\cap_{l\in \z} \cu_{w\ge l}=\ns$.

\item\label{id5} Let $\cu$ and $\cu'$ 
be triangulated categories endowed with
weight structures $w$ and
 $w'$, respectively; let $F:\cu\to \cu'$ be an exact functor.

$F$ is said to be  {\it  right  weight-exact}  (with respect to $(w,w')$) if it maps $\cu_{w\ge 0}$ into $\cu'_{w'\ge 0}$.
We will say that $F$ is {\it  weight-exact}  if it is also {\it  left  weight-exact}, i.e., if   $F(\cu_{w\le 0}) \subset \cu'_{w'\ge 0}$.
 
\item\label{id6} Let $\hu$ be a 
full additive subcategory of a triangulated category $\cu$.

We will say that $\hu$ is {\it negative} (in $\cu$) if
 $\obj \hu\perp (\cup_{i>0}\obj (\hu[i]))$.

\item\label{igen}
We will say that a weight structure $w$ is generated by a class $\cp\subset \obj \cu$ whenever $\cu_{w\ge 0}=(\cup_{i>0}\cp[-i])^{\perp}$.

\end{enumerate}
\end{defi}

\begin{rema}\label{rstws}

1. A  simple 
 example of a weight structure comes from the stupid filtration on the homotopy category $K(B)$ of  cohomological complexes
 for an arbitrary additive category $B$; see Remark 1.2.3(1) of \cite{bonspkar} for more detail. 

2. In the current paper we use the ``homological convention'' for weight structures; 
it was also used 
 \cite{bgern}, \cite{bscwh}, 
   \cite{bkw}, \cite{bososn}, and \cite{bonspkar}, whereas in 
\cite{bws} 
the ``cohomological convention'' was used. In the latter convention 
the roles of $\cu_{w\le 0}$ and $\cu_{w\ge 0}$ are interchanged, i.e., one
considers   $\cu^{w\le 0}=\cu_{w\ge 0}$ and $\cu^{w\ge 0}=\cu_{w\le 0}$. 
 \footnote{Recall also that D. Pauksztello has introduced weight structures independently in \cite{konk}; he called them
co-t-structures.} 
 
Besides, in \cite{bws} both "halves" of $w$ were required to be  additive. Yet this additional restriction is easily seen to follow from the remaining axioms; see Remark 1.2.3(4) of \cite{bonspkar}.

 3. The orthogonality axiom (iii) in Definition \ref{dwstr} immediately yields that $\hw$ is negative in $\cu$.
 We will formulate a certain converse to this statement 
 below.
\end{rema}

Let us recall some basic  properties of weight structures.

\begin{pr} \label{pbw}
Let $\cu$ be a triangulated category 
endowed with a weight structure $w$. 
Then the following statements are valid.
\begin{enumerate}

 \item\label{iort}
 $\cu_{w\ge 0}=(\cu_{w\le -1})^{\perp}$ and $\cu_{w\le -1}={}^{\perp} \cu_{w\ge 0}$.

\item\label{ibougen} 
For any $i\le j\in \z$ the class $\cu_{[i,j]}$ equals the extension-closure of  $\cup_{i\le m\le j}\cu_{w=m}$. 

Moreover, $\cu^b$ equals 
 the subcategory of $\cu$  densely generated by  $\cu_{w=0}$.
 
 \item\label{ifacth}
Let $M\in \cu_{w\le 0}$, $N\in \cu_{w\ge 0}$, and fix some weight decompositions $X_1[1]{\to} M[1]\stackrel{f[1]}{\to}   Y_1[1]$ 
and $X_2\stackrel{g}{\to}   N\to Y$ of $M[1]$ and $N$, respectively. Then $Y_1,X_2\in \cu_{w=0}$ and
any  morphism from $M$ into $N$  can be presented as $g\circ h\circ f$ for some $h\in \cu(Y_1,X_2)$.

\item\label{iwadj} Assume that $w$ is generated by a class $\cp\subset \obj \cu$; let $w'$ be a weight structure on a triangulated category $\cu'$, and let $F: \cu \leftrightarrows \cu': G$  be an adjoint pair of exact functors. Then the following conditions are equivalent:

(i)  $F$ is left weight-exact.

(ii)  $F(\cp)\subset \cu'_{w'\le 0}$.

(iii) $G$ is right weight-exact.

\item\label{iuni} There is at most one weight structure $w_{\cp}$  generated by a given $\cp\subset \obj \cu$, and $\cp\subset \cu_{w_{\cp}\le 0}$ if $w_{\cp}$ exists.

\item\label{iwemb} If $F:\cu\to \cu'$ is a weight-exact embedding, where $\cu'$ is a triangulated category endowed with a weight structure $w'$, then an object $M$ of $\cu$ belongs to  $\cu_{w\le 0}$ (resp. to $ \cu_{w\ge 0}$) if and only if $F(M)\in \cu'_{w'\le 0}$ (resp. $F(M)\in \cu'_{w'\ge 0}$).

\end{enumerate}
\end{pr}
\begin{proof}
The first two assertions were established in \cite{bws} 
(yet pay attention to Remark \ref{rstws}(2)!).
Assertion \ref{ifacth} is precisely  Proposition 1.2.3(9) of \cite{bkw} (and easily follows from the results of \cite{bws} also).

Assertion \ref{iwadj} easily follows from assertion \ref{iort}; see Remark 2.1.5(3) of \cite{bpgws} for more detail.

Assertion \ref{iuni} follows from assertion \ref{iort} easily also; cf.  Remark 2.1.5(1)  of ibid.

The "only if" part of assertion \ref{iwemb} is just the definition of the weight-exactness of $F$. The "if" part follows immediately from assertion \ref{iort}.

\end{proof}

Since in the current paper we are mostly interested in "large" motivic categories, we describe certain properties of weight structures in compactly generated triangulated categories.\footnote{Weight structures described in Proposition \ref{pwexb}(1) were called {\it strongly $\aleph_0$-generated ones} in Remark 4.4.4(1) of \cite{bpgws}. 
Moreover, this weight structure $w$ is said to be {\it 
class-generated by $\hu$} in \cite{bsnew}.} We start with 
a formulation of this sort that will be sufficient for \S\ref{sayoconj}.

\begin{pr}\label{pwexb} 
Assume that $\cu$ is 
compactly generated by its full negative 
additive  subcategory  $\hu$. 

1. Then $\cu$ possesses a (unique) 
weight structure $w$  generated by $\obj \hu$. 
Moreover, $w$ is left non-degenerate, whereas 
 $\hw$  
equals the Karoubi-closure 
of the category of all  coproducts of  objects of $\hu$ (in $\cu$). 

Furthermore, $\cu_{w\ge 0}$ equals the pre-aisle $\cu_+$  generated by $\obj \hu$.

2. Assume that a couple $(\cu',\hu')$  satisfies our assumptions on $(\cu,\hu)$ also. Then any exact functor $F:\cu'\to \cu$ that respects coproducts and sends $\obj \hu'$ into $\obj \hu$ is weight-exact with respect to the weight structure $w'$ generated by $\obj \hu'$ and $w$.


\end{pr}
\begin{proof}
 1. $w$ exists according to Theorem 4.5.2(I) of \cite{bws} (cf. also the proof of  \cite[Theorem 4.3.2(III)]{bws} where the corresponding description of 
$\cu_{w\ge 0}$ is written down explicitly; weight structures of this type are treated in more detail in \cite{bsnew}). 
$w$ is easily seen to be left non-degenerate; this fact follows from (the categorical dual to) Corollary 5.4.1(8) of \cite{bpgws}. 
$\hw$ is calculated in  Theorem 4.5.2(II.1) of  \cite{bws}. 

Next,   Proposition \ref{pbw}(\ref{iort}) implies that $\cu_{w\ge 0}$ is closed with respect to extensions and $[1]$; it is closed with respect to coproducts since objects of $\hu$ are compact. Thus $\cu_{w\ge 0}$ contains $\cu_+$, and the converse inclusion easily follows from Theorem 4.3.2(V.1) of  ibid. combined with Proposition \ref{pbw}(\ref{ibougen}). 



2. The  left weight-exactness of $F$ follows from Proposition \ref{pbw}(\ref{iwadj}).
Its right weight-exactness is immediate from (the "furthermore" part of) the previous assertion.

\end{proof}

In order to treat the general setting of \S\ref{srelmot} we will also need the following statement.

\begin{pr}\label{pwexbb} 
Let $\cu^c$ be 
a triangulated category. Assume that certain $\cu_-,\cu_+\subset \obj \cu^c$ satisfy 
 the axioms (ii) and (iii) of a weight structure (for $(\cu_{w\le 0},\cu_{w\ge 0})$). 
Let $\hu$ be a full negative subcategory of $\cu^c$  that densely generates it and such that for any $N\in 
\obj \cu$ the following conditions are fulfilled: 

(i) for any $i\in \z$ there exists  a 
 distinguished triangle   $X\to N[i]\to Y[1]$ such that $X$ belongs to the 
 envelope  $\cu^c_{w^c\le 0}$ of  $\cu_-$ and $Y$ belongs to the 
envelope $\cu^c_{w^c\ge 0}$ 
of  $\cu_+$;  

(ii) 
There exist $i_1$ and $i_2\in \z$ such that $N[i_1]\in \cu^c_{w^c\le 0}$ and $N[i_2]\in \cu^c_{w^c\ge 0}$.

Then the following statements are valid.

1. $\cu^c_{w^c\le 0}$ and $\cu^c_{w^c\ge 0}$ yield a bounded weight structure for $\cu^c$.

2. Assume in addition that $\cu$ is a triangulated category that contains $\cu^c$ and is compactly generated by it.
Then 
 $\hu=\hw^c$ (along with $\cu$) satisfies all the conditions of the previous proposition. 
\end{pr}
\begin{proof} 1. 
Immediate from 
Theorem 2.1.1(II) of \cite{bonspkar}.

2.  $\hw^c$ is negative in $\cu^c\subset \cu$ by Remark \ref{rstws}(3). Next, $\cu_{w^c= 0}$ densely
generates $\cu$ by Proposition \ref{pbw}(\ref{ibougen}); thus it compactly generates $\cu$.

\end{proof}

\subsection{On 
intersections of "purely compactly generated" subcategories}\label{sloc}

We call a category $\frac A B$ the {\it factor} of an additive
category $A$
by its full additive subcategory $B$ if $\obj \bl \frac A B\br=\obj
A$ and $(\frac A B)(X,Y)= A(X,Y)/(\sum_{Z\in \obj B} A(Z,Y) \circ
A(X,Z))$.

\begin{pr}\label{ploc}
I. Let $\du$ be a 
triangulated subcategory of
$\cu$; suppose
 that $w$ induces a weight structure $w_{\du}$ on $\du$
(i.e., $\obj \du\cap \cu_{w\le
 0}$ and $\obj \du\cap \cu_{w\ge
 0}$ give a weight structure for $\du$).
Assume that  the Verdier quotient of $\cu$
by $\du$ exists and 
denote by $\pi$ the localization functor $\cu\to \cu/\du$.

Then the following statements are valid.

1. $w$ induces a weight structure on  $\cu/\du$, i.e.,  the $\cu/\du$-Karoubi-closures of $\pi(\cu_{w\le
 0})$ and $\pi(\cu_{w\ge 0})$  
give a weight structure $w_{\cu/\du}$ on $\cu/\du$ (and so, $\pi$ is weight-exact with respect to $w$ and $w_{\cu/\du}$).

2. 
The heart ${\underline{Hw}}_{\cu/\du}$ 
of 
 $w_{\cu/\du}$ 
is the Karoubi-closure  of (the natural image of) $\frac { \hw} {\hw_{\du}}$ in
$\cu/\du$.
In particular, for any $M^0,N^0\in \cu_{w= 0}$ the homomorphism $\cu(M^0,N^0)\to \cu/\du(\pi(M^0),\pi(N^0))$ 
 is surjective.

3. For any $M\in \cu_{w\le 0}$ and  $N\in \cu_{w\ge 0}$
the homomorphism $\cu(M,N)\to \cu/\du(\pi(M),\pi(N))$ 
 is surjective.

II. Assume that $\cu$, $\hu\subset \cu$, $F:\cu'\to \cu$, and $\hu'\subset \cu$   are  as in Proposition \ref{pwexb}; assume that the category $\du=\cu'$ is a full subcategory of $\cu$. 

1. Then the weight structure $w$ on $\cu$ given by Proposition \ref{pwexb}(1) restricts to $\du$. 

2. 
The localization $\cu/\du$ exists and is closed with respect to coproducts, the localization functor $\pi:\cu\to \cu/\du$ respects coproducts, 
and $\pi(\obj\hu)$ compactly generates $\cu/\du$.

3. The corresponding category  $\pi(\hu)\subset \cu/\du$ is negative in $\cu/\du$, and the weight structure $w_{\cu/\du}$ given by assertion I.1 coincides with the weight structure $w'$ on $\cu/\du$  generated by  $\pi(\obj\hu)$.\footnote{The existence of the latter weight structure is guaranteed by Proposition \ref{pwexb}(1) provided that $\pi(\hu)$ is negative.}


\end{pr}
\begin{proof}
I. Assertions 1 and 2 are 
 contained in Proposition 8.1.1 of \cite{bws}.

Assertion 3 is an easy consequence of the ("in particular" part of the) previous assertion combined with Proposition \ref{pbw}(\ref{ifacth}). 

II.1. 
Proposition \ref{pwexb}(1) gives a weight structure on $\du$ and part 2 of that proposition ensures that this weight structure is a restriction of $w$. 

2. Immediate from Lemma \ref{lwcg}(1,2,4).

3. According to assertion I.1, $\pi(\obj \hu)$ is contained in the heart of  $w_{\cu/\du}$; hence it is negative (in $\cu/\du$; see Remark \ref{rstws}(3)). Next, $\pi$ is weight-exact with respect to $(w,w')$ according to Proposition \ref{pwexb}(2). 
Applying Proposition \ref{pbw}(\ref{iort}) to $w'$ 
we easily obtain that $w'=w_{\cu/\du}$.

\end{proof}

Now we combine this statement with a result from \cite{bkw}.

\begin{pr}\label{pintersw} 
Let $\cu$ and $\hu$ be as in Proposition \ref{pwexb}(1); let $\hu_1,\hu_2$, and $\hu_3$ be full additive subcategories of $\hu$.
Denote by $\cu_i$  the localizing subcategories of $\cu$ generated by $\hu_i$ (for $i=1,2,3$)  and
assume that any morphism from (an object of) $\hu_1$ into $\hu_2$ vanishes in the Verdier quotient $\cu'$ of $\cu $ by $\cu_3$.\footnote{Note that this quotient exists according to Lemma \ref{lwcg}(1).}

Then  all elements of $\obj \cu_1\cap \obj \cu_2$ become right degenerate in  $\cu'$ (with respect to the 
weight structure $w'$ 
given by Proposition \ref{ploc}(I.1)). 

Moreover, $w$-bounded below elements of $\obj \cu_1\cap \obj \cu_2$ belong to $\obj \cu_3$.
\end{pr}
\begin{proof} 
Denote the localization functor $\cu\to \cu'$ by $\pi$, and
denote the 
weight structures on $\cu_1$ and $\cu_2$ 
 generated by $\obj \hu_1$ and $\obj \hu_2$, respectively (see  Proposition \ref{pwexb}(1)) by 
$w_1$ and $w_2$. 
According to Remark 3.1.6(1) of \cite{bkw},  
 to prove our assertions it suffices to verify that 
the functor $\pi$ kills all $\cu$-morphisms from
$\cu_{1,w_1=0}$ into $\cu_{2,w_2=0}$. Since 
$\pi$ respects coproducts 
 (see Lemma \ref{lwcg}(2)), 
it remains to recall that $\pi$ kills all $\cu$-morphisms from $\hu_{1}$ into $ \hu_{2}$.
\end{proof}

\begin{rema}\label{rintersw}
1. It certainly suffices to assume that   any morphism from an object of $\hu_1$ into $\hu_2$ factors through some object of $\cu_3$ (instead of 
being killed by $\pi$). Actually, these two vanishing conditions are equivalent according to Proposition \ref{ploc}(I.2). 

2. Instead of checking 
these vanishing conditions it certainly suffices to verify that $\pi$ kills all morphisms from the corresponding  $\cu_{1,w^c\le 0}$ into $\cu_{2, w^c\ge 0}$
(see Proposition \ref{pwexbb}). Moreover, Proposition \ref{ploc}(I.3) yields that both of these conditions are equivalent to $\pi(\cu_{1,w^c\le 0}) \perp \pi(\cu_{2, w^c\ge 0})$.
\end{rema}

\section{Intersecting motivic filtrations and a conjecture of Ayoub}\label{sayoconj} 

In this section we intersect  the levels of the slice filtration on motives over a perfect field $k$ with that of the dimension filtration.

In \S\ref{sintersk} we study the intersection of the levels of slice filtration with the dimension one. 

In \S\ref{sapplayoub} we relate our results to Conjecture 4.22 of \cite{ayconj} to obtain several assumptions equivalent to it (we actually prove a more general  
 result of this sort).

In \S\ref{sthom} we prove that our results yield a complete calculation of the 
 intersections in question in the subcategory $\dmers\subset \dmer$ of homotopy $t$-structure bounded above motivic complexes (that was considered in \cite{vbook} and in \cite{1}).

\subsection{Intersecting the dimension and the slice filtrations on unbounded motivic complexes}\label{sintersk}

We start with some preliminaries and notation for motivic complexes.

In this section $k$ will denote a fixed perfect base field of characteristic $p$, and we set $\zop=\z$ if $p=0$.

The set of smooth projective varieties over $k$ will be denoted by $\spv$. 
 
\begin{itemize}
\item
For $R$ being a fixed unital commutative  $\zop$-algebra we consider the $R$-linear motivic categories $\dmger\subset\dmer$.\footnote{In some papers on the subject (in particular,  in \cite[\S4]{degmod}) only the case $R=\z$ is considered; one can easily pass to the case $R=\zop$ or $R$ being a localization of $\zop$ (say, $R=\q$)  using the more-or-less standard "localization of coefficients" techniques (cf. Proposition 5.6.2 of \cite{bpgws}). The reader may certainly restrict himself to these  cases  (that are quite interesting and non-trivial for themselves). One  can also reduce our  results for an arbitrary $R$ to that for the case $R=\zop$; yet this requires some  work on the properties of the "forgetful" functor $\dmer\to DM^{eff}_{\zop}$. So we prefer to  treat the case of a general $R$; the most detailed account on $\dmer$ in this setting is (probably) \cite[\S6]{bev} (cf. also \cite{cdint}, \cite{vbook}, and \cite{bokum}).} 
So, $\dmer$ is the category of unbounded $R$-motivic complexes over $k$ (see Proposition 1.3.1 of \cite{bokum}). It is closed with respect to small coproducts (and so,  Karoubian); it is compactly generated by its triangulated subcategory $\dmger$ of effective geometric motives. Moreover, $\dmger$ is densely generated by  the $R$-linear motives $\mgr(\spv)$ (easy from Theorem 2.2.1(1) of \cite{bzp}; cf. the proof of \cite[Theorem 2.1.2]{bokum}); hence the set $\mgr(\spv)$   compactly generates $\dmer$ also. 
\item
 $\mgr(\spv)$ is a negative subcategory of $\dmer$ (according to Corollary 6.7.3 of \cite{bev}; cf. also Theorem 5.23 of \cite{degdoc}), and the  
Karoubian envelope of  
$\mgr(\spv)$ is the category $\chower$ of effective $R$-linear Chow motives that is also negative in $\dmer$. We will use the symbol $\wchow$ to
denote the weight structure on $\dmer$ 
generated by  $\chower$ (see Proposition \ref{pwexb}(1)). 

\item
We also introduce the following notation: $R\lan 1\ra$ will denote 
the $R$-linear Lefschetz object; so, it equals $R(1)[2]$ in the notation of \cite{1}. 
For $i\ge 0$ and $M\in \obj \dmer$ we will write $M\lan i\ra$ for the object $M\otimes_{\dmer}  (R\lan 1\ra)^{\otimes i}$. 

Recall  that the functor $-\lan i\ra=-\otimes_{\dmer} R\lan i\ra$ is a full embedding of $\dmer$ into itself; thus the essential image
$\dmer(i)=\dmer\lan i\ra$ of this functor is a full subcategory of $\dmer$ that is equivalent to $\dmer$.
\item
Note that $R\lan i\ra$ is a retract of $\mgr(\p^i)$ (for any $i\ge 0$); thus $\chower\lan i\ra\subset \chower$.

\item
Now we define two filtrations for $\dmer$. The so-called slice (or the effectivity) filtration on $\dmer$ is given by  $\dmer(i)$ 
for $i\ge 0$. 
\item
For $m\in \z$ we will write $d_{\le m} \dmer$ for the localizing subcategory of $\dmer$ generated by $\{\mgr(X)\}$ for $X$ running through smooth $k$-varieties of dimension at most $m$ (so, this category is zero for $m<0$). We note that   $d_{\le m} \dmer$ is compactly generated by  $\{\mgr(P)\}$ for $P$ running through smooth projective $k$-varieties of dimension $\le m$ (see Remark 2.2.3 of \cite{bscwh}). 
\end{itemize}

Now we are able to prove the first motivic 
result of this paper. 

\begin{theo}\label{tintersk}
For any $i,m\ge 0$ any element of  $\obj  \dmer(i) \cap \obj d_{\le m} \dmer$ becomes right weight-degenerate (with respect to the weight structure provided by Proposition \ref{ploc})  in the localization $\dmer/d_{\le m-i} \dmer(i)$. 

Moreover, any $\wchow$-bounded below element of   $\obj  \dmer\lan i\ra \cap \obj d_{\le m} \dmer$ 
 belongs to $\obj d_{\le m-i} \dmer\lan i\ra$.
\end{theo}
\begin{proof}
The proof  is an easy application of 
 Proposition  \ref{pintersw}. We take $\cu=\dmer$, $\hu=\chower$, $\hu_1\subset \hu$ being the category
of 
motives of smooth projective varieties of dimension at most $m$, $\hu_2=\chower\lan i\ra$; $\hu_3$ is the category of motives of smooth projective varieties of dimension at most $m-i$ twisted by $\lan i\ra$ (note that $\hu$ contains $\hu_i$ for $i=1,2,3$). 

By the virtue of the aforementioned proposition (cf. also Remark \ref{rintersw}), 
 it suffices to verify that any morphism from $\hu_1$ into $\hu_2$ factors through $\hu_3$. The latter fact is precisely Proposition 2.2.6(2) of \cite{bscwh}.\footnote{Moreover, in the case $\cha k=0$ the proof of \cite[Lemma 3]{gorgul} generalizes to give this statement 
immediately.}  

\end{proof}

\begin{rema}\label{rintersk}

1. Since $R\lan i\ra$ is a retract of $\mgr(\p^i)$,  
  we obviously  have $\hu_3\subset \hu_2$ and $\hu_3\subset \hu_1$.
Hence our 
 theorem describes completely the class of 
$\wchow$-bounded below elements of   $\obj  \dmer(i) \cap \obj d_{\le m} \dmer$.

2. Recall that any compact object of $\dmer$ is $\wchow$-bounded; hence one can apply the "moreover" part of our proposition to the calculation of $\obj  \dmger(i) \cap \obj d_{\le m} \dmger$. 

 The latter calculation has found important applications in \cite{bscwh}. 
For this reason we explain how to avoid using  (Remark 3.1.6(1) of) \cite{bkw} in its proof (however, our argument is rather similar to that in loc. cit.). 
So, we verify under the assumptions of Proposition \ref{pintersw} that any $w$-bounded object of  $\obj \cu_1\cap \obj \cu_2$  belongs to $\obj \cu_3$. 

This argument relies on the theory of weight complexes as introduced in \S3 of \cite{bws} (whereas in \S2.2 of \cite{bpgws} some parts of the theory were exposed more carefully).  We recall that to any object $M$ of $\cu'$ (for $\cu'=\cu/\cu_3$) there is associated its weight complex $t(M)\in \obj K(\hw')$; $t(M)$ is well-defined up to homotopy equivalence. 
The definition of $t(-)$ easily implies that for $i=1$ or $2$ and $M\in \pi(\obj \cu_i)$ the complex $t(M)$ is homotopy equivalent to a complex whose terms belong  
to $\pi(\cu_i{}_{w_i=0})$. 
Since $\pi(\cu_1{}_{w_1=0})\perp \pi(\cu_2{}_{w_2=0})$ (see Remark \ref{rintersw}), 
 for any $N\in \obj \cu_1\cap \obj \cu_2$ the morphism $\id_{t(\pi(N))}$ is zero in  $K(\hw')$.
Applying  Theorem 3.3.1(V) of \cite{bws} we conclude that $\pi(N)=0$.

3. Note that in 
our theorem one cannot replace $\dmer\subset \dmr$ by the corresponding version of the motivic stable homotopy category $\sht(k)$ (say, for $R$ equal to $ \z$ or to $\z[\sss\ob]$ for $\sss$ being a set of primes; then one can define the $R$-linear version of $\sht(k)$ as a certain localization).  
   One of the reasons for this is that there is no Chow weight structure on $\sht^c(k)\subset \sht(k)$ (and on $\sht^c(k)[\sss\ob]$ if $2\notin \sss$; see Remark 3.1.2 of \cite{binf} and Remark 6.3.1(3) of \cite{bgern}). Moreover,  even the "compact version" of  Theorem  \ref{tintersk} does not carry over to the $\sht(k)$-setting.
There probably exist plenty of examples illustrating the latter statement. Here we will only note that  for $k$ being any formally real field the corresponding category $d_{\le 1}\she(k)$
contains a non-zero compact  infinitely effective object (i.e., an element of  $\cap_{i\ge 0}\obj\she(k)(i)$). Indeed, for the object $C$ constructed in Remark 2.1.2(3) of \cite{binf}  we surely have $C(1)\in \obj d_{\le 1}\she(k)$ and $C\neq 0$ in $\she(k)[\sss\ob]$ unless $3\in \sss$ (and $3$ may be replaced by any other odd prime here). Yet  the associated motif $M_{k}(C)$ of $C$ is zero by loc. cit.; hence $C$ is infinitely effective in $\she(k)$ by Theorem 3.1.1 of ibid.

\end{rema}

\subsection{An application 
 to a conjecture of J. Ayoub}\label{sapplayoub}

We recall some basics on "slice" functors.

For any $i\ge 0$ the right adjoint   to the 
functor $-\lan i \ra:\dmer\to \dmer$ can certainly be described as $\ihom_{\dmer}(R\lan i\ra,-) $;
 this functor respects small coproducts.
Next, the composition $\nu^{\ge i}=\lan i \ra \circ \ihom_{\dmer}(R\lan i\ra,-)  :
\dmer\to \dmer$ equals the composition of the embedding $\dmer \lan i \ra\to \dmer$ with the right adjoint to it.

Now we establish some new properties of the slice functors.

\begin{pr}\label{rayouconjred}
Fix $i,m\ge 0$ (along with $R$).
Then the following statements are valid.

I. $\nu^{\ge i}$ is right 
$\wchow$-exact. 

II. The following conditions are equivalent.

1. $\nu^{\ge i}$ sends $d_{\le m}\dmer$ into itself.

2. $\ihom_{\dmer}(R\lan i\ra,-) $ sends $d_{\le m}\dmer$ into $d_{\le m-i}\dmer$. 

3. $\obj \dmer\lan i\ra$ becomes orthogonal to 
$\obj d_{\le m}\dmer$   in the localization $\dmer/d_{\le m-i}\dmer\lan i \ra$. 

4. For any smooth projective $P,Q/k$ and $n\in \z$ with $\dim Q\le m$ the image of $\mgr(P)\lan i \ra$ in   $\dmer/d_{\le m-i}\dmer\lan i \ra$ is orthogonal to (the image of)  $\mgr(Q)[n]$.

5. For any $P,Q,n$ as above  the image of $\mgr(P)\lan i \ra$ in   $\dmger/d_{\le m-i}\dmger\lan i \ra$ is orthogonal to (the image of)  $\mgr(Q)[n]$.

\end{pr}
\begin{proof}
I. 
Recall that $\chower\lan i\ra \subset \chower$ (see \S\ref{sintersk}); hence the twist functor $-\lan i \ra:\dmer\to \dmer$ is weight-exact according to Proposition \ref{pwexb}(2).
Applying Proposition \ref{pbw}(\ref{iwadj}) to the adjunction  $-\lan i \ra \dashv \ihom_{\dmer}(R\lan i\ra,-) $ we obtain that $\ihom_{\dmer}(R\lan i\ra,-) $ is  right 
$\wchow$-exact. It remains to note that the composition of right weight-exact functors is right  weight-exact also.


II. Condition II.2 implies condition II.1  due to the weight-exactness of $-\lan i \ra:\dmer\to \dmer$; cf. the proof of assertion I. 

Next, recall that $d_{\le m}\dmer$  is  generated by  $\{\mgr(P)\}$ for $P$ running through smooth projective $k$-varieties of dimension $\le m$, as a localizing subcategory of $\dmer$. Hence to verify the converse implication it suffices to check whether condition II.1 implies 
 that  $\ihom_{\dmer}(R\lan i\ra,\mgr(P))\in d_{\le m-i}\dmer$ if $P$ is smooth projective of dimension at most $m$. Now, $\nu^{\ge i}(\mgr(P))\in \dmer{}_{\wchow\ge 0}$ according to assertion I. It remains to apply (the "moreover" part of) Theorem  \ref{tintersk}.

So, the first two conditions are equivalent to the assumption that  \break $\nu^{\ge i}(d_{\le m}\dmer)\subset d_{\le m-i}\dmer\lan i \ra$. 
The latter assertion is equivalent to condition II.3 by Lemma \ref{ladj}(3). 

Next, Lemma \ref{lwcg}(2) allows us  to verify the orthogonality  in condition II.3 only for the images  $\pi(\mgr(P))$ and $\pi(\mgr(Q)\lan i \ra[n])$ for $P,Q$, and $n$ as in condition II.4. Hence condition II.3 is equivalent to   II.4 by Lemma \ref{lwcg}(3). This (part of the) lemma also implies that condition II.4 is equivalent to condition II.5.

\end{proof}

\begin{rema}\label{rcayoub}
Note that in the case where $R$ is a $\q$-algebra and $i=1$ our condition II.2 is exactly Conjecture 4.22 of \cite{ayconj}. Certainly (for any fixed $R$) if 
we consider this condition for all $m\ge 0$ then the case $i=1$ of it implies all the other cases. Recall also that certain cases of our condition II.2 
  were verified in Proposition 4.25 of ibid.
\end{rema}

\subsection{Computing intersections inside Voevodsky's $\dmers$}\label{sthom}

Now we extend the "moreover" part Theorem \ref{tintersk} to a wider class of objects. We start from a few remarks.

\begin{rema}\label{rdeg}
1. The problem with our arguments is that weight structures 
 do not say much on (right) weight-degenerate objects. Note here that non-zero right 
 $\wchow$-degenerate objects in $\dmer$ do exist (at least) whenever $k$ is a big enough field and $R$ is not a torsion ring (see Remark 2.3.5(3) of \cite{bkw} that relies on Lemma 2.4 of \cite{ayconj}).

We also note that  this Ayoub's motif 
belongs to $\dmer{}^{\thom\le 0}$ (see below), is {\it infinitely effective} (i.e., belongs to $\obj \dmer(r)$ for all $r\ge 0$), and its {\it Betti realization} vanishes (as proved in loc. cit.). 

The author suspects that all $\wchow$-degenerate objects of $\dmer$ are infinitely effective.  

2. Starting from the first motivic papers of Voevodsky one of the main tools of working with motivic complexes was the so-called homotopy $t$-structure  $\thom$. Actually, instead of the unbounded category $\dmer$ he essentially considered (see \S14 of \cite{vbook}; the case $R=\z$ was treated in \cite{1}) its $\thom$-bounded above  subcategory $\dmers$ whose objects are $\cup_{i\in \z}\dmer{}^{\thom\le i}$ (so, we use the cohomological convention for $t$-structures here; cf. \cite[\S4.1]{bws} or \cite[\S1.4]{bkw} for more detail on it). Thus the 
intersection result that we will prove
 below is completely satisfactory from this older point of view.

Now we recall a description of $\thom$ that will be convenient for our purposes. According to Theorem 2.4.3 and Example 2.3.5(1) of \cite{bondegl} (where the assumptions on $R$ are the same as in this paper, but the convention for $t$-structures is the homological one), $\dmer{}^{\thom \le 0}$ is the pre-aisle generated by $\cup_{i\ge 0}\obj \chower\lan i \ra[-i]$. Certainly, $\dmer{}^{\thom \ge 0}$  can be recovered from  $\dmer{}^{\thom \le 0}$ using the orthogonality condition (still we will not use this fact below).
\end{rema}

We will need the following very useful lemma.

 \begin{lem}\label{layoub}
For any $m\ge 0$ we have $\obj \dmer (m+1)\perp \obj d_{\le m}\dmer$.

\end{lem}
\begin{proof}
In the case where $R$ is a $\q$-algebra this statement is essentially contained in Proposition 4.25 of \cite{ayconj}. 
 Moreover, most of the arguments used in the proof of 
 loc. cit. can be easily carried over to the case of a general ($\zop$-algebra) $R$; this yields a reduction of our assertion to the case $m=0$. In the latter case the assertion is  quite simple; one can easily deduce it from the vanishing of higher Chow groups of negative codimensions for any smooth variety (see Corollary 6.1 of \cite{bev}). 
 
\end{proof}

Now we are able to 
prove the following 
  version of  Theorem \ref{tintersk}. 

\begin{pr}\label{pintersthom}
For any $i,m\ge 0$ 
the class $\obj \dmers\lan i \ra\cap \obj d_{\le m}\dmer $ 
 equals $(\obj \dmers\cap \obj d_{\le m-i}\dmer)\lan i \ra $. 
\end{pr}
\begin{proof}
Certainly, if $N\lan i \ra\in\obj  \dmers$ for $N\in \obj \dmer$ then $N$ is $\thom$-bounded above also (an easy well-known fact; cf. \cite[Corollary 3.3.7(3)]{bondegl}). 
Thus it suffices to prove that any $M\in \dmer{}^{\thom \le 0}\cap \obj \dmer\lan i \ra \cap \obj d_{\le m}\dmer $ 
 belongs to $\obj d_{\le m-i}\dmer\lan i \ra $. 

Now we apply Lemma \ref{ladj}(2). We take $\cu=\dmer$, $\eu=\dmer\lan m+1\ra$, $\cu'=\cu/\eu$. By Lemma \ref{layoub}, 
 $\obj \eu \perp  \obj d_{\le m}\dmer $. Hence the localization functor $\pi:\cu\to \cu'$  restricts to a full embedding of  $d_{\le m}\dmer $ into $\cu'$.
Since  $ d_{\le m-i}\dmer\lan i \ra\subset  d_{\le m}\dmer $, 
 it suffices to verify that $\pi(M)$ 
 belongs to 
$\pi(\obj d_{\le m-i}\dmer\lan i \ra)$. 

Now, the description of  $ \dmer{}^{\thom \le 0}$ given in Remark \ref{rdeg}(2) yields that $\pi(M)$ belongs to the pre-aisle 
 generated by  $\cup_{0\le i\le m}\pi(\obj \chower)\lan i \ra[-i]$. Thus $\pi(M)$ is bounded below with respect to the 
 weight structure for $\cu'$ generated by $\pi(\obj \chower)$ (see Proposition \ref{ploc}(II)). 
Hence the image  $M'$ of $\pi(M)$ in the localization $\cu'/\pi(d_{\le m-i}\dmer)\lan i \ra)$\footnote{Actually, $M'$ is just $M$ considered as an object of $\cu'/\pi(d_{\le m-i}\dmer \lan i \ra)$.} is also bounded below  with respect to the corresponding weight structure provided by Proposition \ref{ploc}  (here we apply parts II.(2--3) of the proposition). 
 On the other hand, 
$M'$ is  also right 
   $\wchow$-degenerate since the image of $M$ in $\cu/(d_{\le m-i}\dmer\lan i \ra)$ is so by Theorem \ref{tintersk} (here we invoke  
 Proposition \ref{ploc}(II.3) once again). Thus $M'\perp M'$ by the orthogonality axiom of weight structures; hence $M'=0$. Therefore $\pi(M)$ does 
 belong to $\pi(d_{\le m-i}\dmer\lan i \ra)$.

\end{proof}

\begin{rema}\label{rtbound} 1. Formally (the formulation 
of) our proposition "depends on $\dmer$" since it treats  intersections of certain localizing subcategories of $\dmer$ with $\obj\dmers$. 
However, the 
 natural $\dmers$-version of our result (
stated in terms  the corresponding "localizing" classes of objects in $\dmers$; those are closed only with respect to those coproducts that exist in $\dmers$)
 is also true. Indeed, for  any $j\ge 0$  there exists an (exact) right adjoint to the embeddings $d_{\le j}\dmer \to \dmer$ (see Lemma \ref{lwcg}(1)).  Since this functor also respects coproducts (see part 2 of the lemma), it suffices to check that it restricts to an endofunctor of $\dmers$. The latter assertion was mentioned in  the beginning of \cite[\S3.4]{1} (in the case $R=\z$ that does not differ from the general one in this matter); respectively, it  can also be easily established using the methods of the proof of \cite[Corollary 3.7]{bondegl}. 

2. Note that the Chow weight structure on $\dmer$ can be restricted to $\dmers$ (since the arguments of \cite[\S7.1]{bws} carry over to our more general 
setting without any problems). Yet the author was not able to apply this fact for the purposes of the current paper.

3. The main subject of \cite{bokum} is an interesting candidate for the Chow weight structure on $\dmer$ that is defined
independently from the assumption that $p$ is invertible in $R$. It is proven that this weight structure satisfies several important properties of $\wchow$ (in particular, it coincides with $\wchow$ whenever $R$ is a $\zop$-algebra). Yet it is not clear whether this weight structure may be restricted to the subcategories $d_{\le m}\dmer$ (unless it coincides with $\wchow$); so the methods of the current paper cannot be used for the study of $\obj  \dmer \lan i \ra \cap \obj d_{\le m} \dmer$ in this greater generality.     

\end{rema}

\section{A generalization to motives over general base schemes}\label{srelmot}

This section is dedicated to the study of certain versions of the slice and the dimension filtrations for motives over a base. Our methods work for all motivic categories
satisfying a certain list of axioms; we also describe three series of examples satisfying them.

In \S\ref{sdimf} we (essentially) generalize Gabber's notion of a dimension function; this allows to apply the results of this section to motives over any Noetherian separated excellent scheme of finite Krull dimension.

In \S\ref{srmotfilt} we describe some of the axioms on motivic categories that are necessary to define  our filtrations. 

In \S\ref{saddort} we recall (from \cite{bondegl}) some properties of (co)niveau spectral sequences that converge to the cohomology of Borel-Moore objects. They enable us to relate certain 
vanishing properties of the "$\md$-motivic cohomology" of fields to  more general ortogonality statements; we obtain quite interesting vanishing statements this way.

In \S\ref{swchow} we prove that the levels of our filtrations are endowed with certain Chow weight structures. This enables us to "intersect two types of filtrations" and obtain the corresponding analogue of Theorem  \ref{tintersk}. We also mention three (series of) examples of motivic categories satisfying our assumptions.

\subsection{On generalized dimension functions}\label{sdimf}

Since we want our results to be valid for a wide range of schemes, we will need a certain ``substitute'' of the Krull dimension function $\dim(-)$.\footnote{The reason is that we want some notion of dimension that would satisfy the following property: if $U$ is open dense in $X$ then its ``dimension'' should equal the ``dimension'' of $X$.}
So we need certain dimension functions 
 that 
are somewhat more general than 
 Gabber's 
ones (as introduced in \S XIV of \cite{illgabb} and applied to motives in \cite{bondegl}). 
So we "axiomatize a little" the construction described in \S3.1 of \cite{bonegk} (cf. also \S4 of \cite{bonspkar}).
Recall that a     localization of a scheme  of finite type over some base $X$ is said to be essentially of finite type over $X$.

\begin{defi}\label{ddf}
Let $B$ be a Noetherian separated excellent scheme of finite Krull dimension. 

\begin{enumerate}
\item\label{ddf-1} Throughout this section we will say that a scheme is a   $B$-scheme only if it is separated and of essentially  of finite type over $B$.
The class of $B$-schemes will be denoted by $G=G(B)$. 
 All the morphisms mentioned below will be separated $B$-morphisms; we will use the symbol $\gu$ to denote the corresponding category of schemes. 

We will call the spectra of fields that are essentially of finite type over  $B$ just {\it $B$-fields}.

\item\label{ddf-2}  Let $\de^B$ be a 
function from the set $\bb$ of Zariski points of  $B$ into non-negative integers 
 that satisfies the following condition: if  $b\in \bb$ and a point $b'\in \bb$ belongs to its closure  then $\de(b)\ge \de(b')+\codim_{b}b'$. 

Then we extend $\de^B$ to 
$B$-fields as follows: for 
a $B$-field $y$ 
and $b$ being its image in $B$
we set $\de(y)=\de^B(y)=\de^B(b)+\operatorname{tr.\,deg.}k(y)/k(b)$, where $k(y)$ and $k(b)$ are the corresponding  fields.

\item\label{ddf-3}  For $Y$ being an $\sz$-scheme we define $\de(Y)$ as the maximum over points of $Y$ of $\de(y)$; we will sometimes call $\de(Y)$ (resp. $\de(y)$) the {\it $\de$-dimension} of $Y$ (resp. of $y$).

\item\label{ddf-4}   Denote by $B^k$ the irreducible components of $B$. 
Then for  any $B^{k_0}\in \{B^k\}$ and  a Zariski point $b\in B^{k_0}$ we set  $\de^{k_0}(b)=-\operatorname{codim}_{B^{k_0}}b$. 

\end{enumerate}
\end{defi}
Now we describe the main properties of $\de$ and its relation to $\{\de^k\}$.

\begin{pr}\label{pdimf}
Let $X$ and $U$ be $B$-schemes, $c>0$.
Then the following statements are valid.

\begin{enumerate}
\item\label{pdf-1}
For $B^k$ running through irreducible components of $B$ take a set of integers $c^k$ satisfying the condition $c^k\ge \dim B^k$ (for all $k$). Then in Definition \ref{ddf}(\ref{ddf-2}) one can take
$\de^B:\bb\to \n$ defined as follows: $\de^B(b)$ is the minimum of  $c_k+\de^k(b)$ for all $k$ such that $b\in B^k$.

\item\label{pdf-ast}
If $x$ and $x'$ are points of $X$ and 
$x'$ belongs to the the closure of $x$ then 
 $\de(x)\ge \de(x')+\codim_{x}x'$.

\item\label{pdf-2}
$\de(X)$ equals the maximum of the $\de$-dimensions of generic points of $X$. 
In particular, it is a finite integer. Moreover, $\de(X)\ge \dim X$. 

\item\label{pdf-3}
 $\de(U)\le \de(X)+d$ whenever there exists a  
 $B$-morphism $u\colon U\to X$ generically of dimension at most $d$. 
Moreover, we have an equality here whenever $d=0$ and  $u$ is  dominant, and a strict inequality if 
 the image of $u$ is nowhere dense. 

\item\label{pdf-4}
Furthermore, if 
$U\subset X$ and any irreducible component of $U$ is of codimension at least $c$ in some irreducible component of $X$ (containing it) then $\de(U)\le \de(X)-c$.

\end{enumerate}

\end{pr}
\begin{proof} 
\ref{pdf-1}. Obvious.

\ref{pdf-ast}. We choose a connected component $B^0$ of $B$ such that  $x$ (and so also $x'$) lies over it; we can certainly assume that $X$ lies over $B^0$ also. Now, $\delta^0$ yields a {\it dimension function}  on $B^0$ in the sense of \cite[Definition XIV.2.1.10]{illgabb}.
We can "extend" it to $X$ using Corollary 2.5.2 of ibid.; so we have $\de^0(x)=\delta^0(b)+  \operatorname{tr.\,deg.}k(x)/k(b)$ and  $\de^0(x')=\delta^0(b')+  \operatorname{tr.\,deg.}k(x')/k(b')$ for $b$ and $b'$ being the images in $B^0$ of $x$ and $x'$, respectively. By the definition of a dimension function, we have 
$\de^0(x)-\de^0(x')\ge \codim_{x}x'$. Combining this equality with the definition of $\de$ (including $\de(b)\ge \de(b')+\codim_{b}b'=\de(b')+\de^0(b)-\de^0(b')$) 
we get the result.

\ref{pdf-2}
--\ref{pdf-4}. 
Assertion \ref{pdf-ast}  
 certainly implies that $\de(X)$ equals the maximum of the $\de$-dimensions of generic points of $X$. Moreover, 
combining this assertion with the non-negativity of $\de$ we immediately obtain that $\de(X)\ge \dim X$. So assertion \ref{pdf-2} is proved.
Along with the definition of  (the "extended version" of) $\de$ and with assertion \ref{pdf-ast} this assertion easily implies   assertion \ref{pdf-3}.
Assertion \ref{pdf-4} 
easily follows from  assertion \ref{pdf-ast} also.
\end{proof}

\begin{rema}\label{rdimf}
 1. If $B$ is of finite type over a field or over $\spe \z$ then 
one can take $\de(Y)=\dim Y$ for any $Y$ that is of finite type over $B$. 
Thus the reader satisfied with this restricted setting may replace the $\de$-dimensions 
 of all finite type $B$-schemes 
 in this section by their Krull dimensions. 


2. Another special case is the one of a "true" (Gabber's) dimension function, i.e., of the one for which the inequality in Proposition \ref{pdimf}(\ref{pdf-ast}) becomes an equality. We do not treat  functions of this type in detail since their existence imposes a certain restriction on $B$. However, for some of the arguments below we will need the following function on the points of a $B$-scheme $X$ all of whose connected components are irreducible: $\de'(x)=\dim X-\codim_X x$ for all $x\in \xx$.

\end{rema}

\subsection{On relative motives and filtrations for them}\label{srmotfilt}

In this section we will demonstrate that our methods can be applied to various relative motivic categories (i.e., to motives $\md(B)$ over the base scheme $B$ as above). We will not specify the choice of $B$ and of the motivic category $\md$ here (except in Remark \ref{rexamples}(1) below). Instead we will only 
 describe those properties of $(B,\md)$ that we need and give references to other ones (mostly to \cite{cd}; most of the definitions we need may also be found in \cite{bondegl} and \cite{bonegk}). 


Similarly to the previous section, our main tool is the existence of a certain Chow weight structure for $\md(\sz)$.
Note  that the in case where $\md$ is a certain version of the Voevodsky's motives functor the corresponding Chow weight structures were treated in detail in \cite{hebpo}, \cite{brelmot}, and \cite{bonivan} (cf. also \cite{bonegk}). In these papers one can find concise lists of (the corresponding versions of) the properties
of $\md$ that we will need below along with some analogues of the arguments that we will  apply in the current section.

 From now on we will assume that $\md$ is a
{\it motivic triangulated category}
(see Definition {2.4.45} of \cite{cd}). So we fix some $\sz$ as above and
assume that $\md(S)$ is defined whenever $S\in G$. 
In particular, this means that $\md$ is 
 a $2$-functor  from 
the 
category $\gu$ (see Definition \ref{ddf}(\ref{ddf-1})) 
 into the $2$-category of tensor triangulated categories that are closed with respect to small coproducts.    Moreover, we will assume that $\md$  is   oriented (see  Remark 13.2.2  and Example 12.2.3(3) of ibid.). 
For any $S\in G$ we will write $\oo_S$ for  the (tensor) unit object of $\md(S)$.

For convenience, we note that any scheme that is of finite type (and separated) over a $B$-scheme belongs to $G$ itself by the well-known properties of "continuity" of the set of schemes of finite type over a base; yet we will not actually need the general case of this statement below.


To define our filtrations on a subcategory $\du$ of $\md(B)$ 
we start from defining {\it Borel-Moore objects} following  \cite{bondegl}. If   $f\colon Y\to S$ 
is a finite type separated morphism of $B$-schemes 
then we set $\mgbms(Y)=f_!(\oo_Y)$.  

Now we describe our filtrations. Instead of the "usual" Tate (or "Lefschetz") twists $-(i)$ it will be somewhat more convenient for us (similarly to the previous section) to consider $-\lan i\ra=-(i)[2i]$ for $i\in \z$. Recall also that  twists "commute with" all the functors of the type $g^*$, $g_*$, $g^!$, and $g_!$.

Denote by $\du$  the  "$\delta$-effective subcategory" of $\md(\sz)$ that we define following Definition 2.2.1 of \cite{bondegl} (that is  closely related to the classes of morphisms $B_n$ and other
constructions from \S2 of \cite{pelaez}). So, $
\du$ is the   localizing subcategory of $\md(B)$ generated by $\mgbmsz(Y)\lan \delta(Y)\ra$ for $Y$ running through all  separated finite type  $\sz$-schemes. 
Note that $\du\lan i \ra\subset \du$ for any $i\ge 0$ (cf. Remark \ref{rintersk}(1)).

Next, for $m\in \z$ 
 we define $\de_{\le m}\du$ as the localizing subcategory of $\du\subset \md(\sz)$ generated by $\mgbmsz(Y)\lan \delta(Y)\ra$ for $Y$ running through all  separated finite type  $B$-schemes
such that $\delta(Y)\le m$. 

\begin{rema}\label{rfield}
As we will explain below, 
one can take $\md=\dmr(-)$ for $B$ being the spectrum of any perfect field (of characteristic invertible in $R$). 
If we also take $\delta$ being the Krull dimension for all schemes of finite type over $B$ (see Remark \ref{rdimf}(1)) then $\du=\dmer$ and the filtrations 
defined in this section are exactly the ones described in \S\ref{sintersk}; this is an easy consequence of Corollary 2.3.11 of \cite{bondegl}.

\end{rema}

Now we introduce some 
more restrictions on $\md(-)$.  Essentially  following \cite{bondegl} and Definition 1.3.16 of \cite{cd}, we require $\md$ to be 
{\it compactly generated by its Tate twists}, i.e., $\md(S)$ should be compactly generated by $\mgbms(Y)\lan i\ra$ for $Y$ running through all  separated finite type  $S$-schemes and $i\in \z$ whenever $S\in G$. 

We will assume that $\md$ satisfies the following {\it absolute purity} property: if $i:Z\to X$ is a closed embedding of  regular $G$-schemes everywhere of codimension $c$ then the object $i^!(\oo_X)$ is isomorphic to $\oo_Z\lan -c \ra$ (yet we don't have to assume the existence of canonical isomorphisms of this sort).

\begin{rema}\label{rpur}
1. Recall that under our assumptions 
this isomorphism implies  the existence of a distinguished triangle \begin{equation}\label{egysc} z_* R_Z\lan -c \ra\to x_* R_X\to u_*R_U
\end{equation}
 if $U=X\setminus Z$, $x:X\to S$ is a separated morphism, $S\in G$, and $u,z$ are the corresponding composition morphisms. Certainly, if $\alpha$  is a regular stratification of a regular scheme $X$ (i.e., it is a presentation of $X$ as $\cup X_l^\alpha$, where $X_l^\alpha$, $1\le l\le n$, are pairwise disjoint locally closed regular subschemes of $X$ and  each $X_l^{\alpha}$ is open in $\cup_{i\ge l} X_i^{\alpha}$)\footnote{This somewhat weak notion of a stratification was used in some previous papers of the author.} and all $X_l^\alpha$ are connected then these triangles yield that for any $B$-morphism 
 $g:X\to S$ the object $g_*(\oo_X)$ belongs to the envelope of $g^\alpha_{l*}\oo_{X^\alpha_l}\lan c_l^\alpha\ra$, where  $c_l^\alpha$ are the corresponding ("local") codimensions.

2. Since $j^!=j^*$ if $j$ is an open embedding, we actually have   $i^!(\oo_X)\cong \oo_Z\lan -c \ra$ for $i$ being an arbitrary embedding of regular schemes (everywhere) of codimension $c$. 
\end{rema}

So (regular; yet cf. Remark \ref{rpurbm} below) $B$-schemes can be "decomposed into pieces" in certain cohomological statements. In order to "pass to points of schemes" we will also require $\md$
to fulfil the {\it continuity} property (see \S4.3 of \cite{cd}) that we will now define.

 It will be 
 sufficient for our purposes to apply Definition 4.3.2 of loc. cit. directly, i.e., for  a $B$-scheme  $X$  being the (inverse) limit of an  affine filtering projective system of  schemes $X_\be$ for $\be\in B$ and $N\in \obj \md(X_{\be_0})$ for certain
 $\be_0\in B$ we need the morphism group from $\oo_X$ into $j_{\be_0}^*N$ to be isomorphic to $\inli_{\be\ge \be_0}\md(X_\be) (\oo_{X_\be}, j_{\be\be_0}^*(N))$, where $j_{\be\be_0}:X_\be \to X_{\be_0}$ and $j_{\be_0}:X\to X_\be$ are the corresponding transition morphisms. 

\subsection{Some additional orthogonality assumptions and their consequences}\label{saddort}

We recall some of the properties of Borel-Moore objects.

\begin{pr}\label{pbm}
Assume that $X \in G$, $Y$ is a 
(separated) finite type $X$-scheme, and $i:Z\subset Y$ is a closed embedding.
Then the following statements are valid.

1. 
 Denote by $j:U\to Y$ the complementary open embedding. 
 Then  there exists a distinguished triangle 
$\mgbmx(U) {\to}\mgbmx(Y) 
{\to} \mgbmx(Z).$ 

2. If $Z$ is the reduced scheme associated to $Y$ then $\mgbmx(Y) \cong \mgbmx(Z)$.

3. If $g:X'\to X$ is a separated morphism of $B$-schemes then we have $g^*\mgbmx(Y)\cong \mgbm_{X'}(X'\times_X Y)$.

\end{pr}
\begin{proof}
Assertions 1 and 3 are contained in \S1.3.8 of \cite{bondegl}.

Assertion 2 can be deduced from assertion 3; it is also an immediate consequence of Proposition 2.3.6(1) of \cite{cd}.

\end{proof}

\begin{rema}\label{rpurbm}
Part 1 of the proposition certainly yields that
 the object $\mgbms(X)$  belongs to the envelope of 
$\mgbms(X^\alpha_l)$ whenever   $\alpha$  is a (not necessarily regular)   stratification of a $B$-scheme $X$ (cf. Remark \ref{rpur}(1))
and $g:X\to S$  
finite type $\gu$-morphism. 

Along with part 2 of the  proposition this observation immediately yields that for any $S\in G$ the category $\md(S)$ is compactly generated by $\mgbms(Y)\lan i\ra$ for $i\in \z$ and $Y$ running through all  separated finite type {\bf regular} $S$-schemes (only). 
Moreover, similar results hold for the categories $\du$ and $\de_{\le m}\du$ described in  \S\ref{srmotfilt}.

\end{rema}

In some of the arguments below we will also apply certain (co)niveau spectral sequences for the cohomology 
$\mgbms(X)$. So now we "translate" the corresponding facts from \cite{bondegl} into  cohomological notation 
and make some computations in the case where $X$ is regular. 

\begin{pr}\label{pcss}
Let $S\in G$, $x:X\to S$ be a finite type (separated) morphism. Then for 
  $m\in \z$ and a cohomological functor $H:\md(S)\to \ab$ 
we define $H_{BM,S}^m(X)$ as $H(\mgbms(X)[-m])$; for $F$ being a Zariski point of $X$ we define  $H_{BM,S}^m(F)$ as $\inli H(\mgbms(F_i)[-m])$, where 
we choose affine subschemes $F_i$ of $X$  so that $F$  is their (filtered inverse) limit, the connecting morphisms 
 $j_{i_1i_2}: F_{i_1}\to F_{i_2}$ are open dense 
embeddings;   the corresponding morphisms in the direct limit 
are induced by  the counits of the adjunctions $j_{i_1i_2,!}\dashv j_{i_1i_2}^*$\footnote{ Recall that it was shown in ibid. 
that  $\inli H(\mgbms(F_i)[-m])$ does not depend on the choice of $F_i$.}. 

Then the following statements are valid.

1. There exists a convergent coniveau (or niveau) spectral sequence  $T$ with 
$E_1^{pq}=\bigoplus_{F\in \xx^p} H_{BM,S}^{p+q}(F),$
 where $\xx^p$ is the 
the set of those points of $X$ 
such that $\delta(X)-\delta(F)=p$; this spectral sequence converges to $E_\infty^{p+q}=H_{BM,S}^{p+q}(X)$.

The induced filtration $F^jH^*$ of $H^*_{BM,S}(X)$ is the $\delta$-coniveau one, i.e., $F^jH^*$  (for a fixed $j\ge 0$)  equals $\cup \ke (H_{BM,S}^*(X)\to H_{BM,S}^*(U_{j}^\al))$ for  $U_{j}^\al$ running through all open subschemes such that $\delta(X)-\delta(X\setminus U_{j}^\al)\ge j$.

2. Assume that $X$ is regular 
and that  $H$ is the functor represented by $y_*\q_Y\lan r\ra$ for some $r\in \z$ and $y:Y\to X$ 
being a finite type separated morphism. Then for this spectral sequence $T$ we have 
$$E_1^{pq}\cong \bigoplus_{F\in \xx^p}   \md(F)(\oo_{F},j_F^*x^!y_*(\oo_Y)\lan r-c_F \ra[p+q]), $$ where $j_F:F\to X$ are the corresponding pro-embeddings, $c_F$ is the codimension of $F$ in $X$.


Moreover, if the connected components of $X$ are irreducible then in both of these statements one can use the function $\de'$ (see Remark \ref{rdimf}(2)) instead of $\de$.
\end{pr}
\begin{proof}
1. The "$\delta'$-version" of assertion is immediate from the  results of  \cite[\S3.1.4]{bondegl}, and the "$\de$-version" is just slightly different from it; so the corresponding the corresponding arguments of loc. cit. yield its proof without any difficulty. 

2. We will prove the "$\de$-version" of this assertion; the proof of the $\delta'$-version is quite the same. 

Fix some $F\in \xx^p$ and denote the (pro)-embeddings $F_i\to X$ (resp. $F\to X$) by $j_i$ (resp. by $j_F$).
Then applying the corresponding properties of oriented motivic categories 
 we obtain that the direct summand of $E_1^{pq}$ corresponding to $F$ is isomorphic to $$\begin{gathered}  \inli_{i\in I} \md(S)(x_!j_{i!}\oo_{F_i},y_*(\oo_Y)
\lan r\ra[p+q]) \cong   \inli_{i\in I} \md(X)(j_{i!}\oo_{F_i},x^!y_*(\oo_Y)
\lan r\ra [p+q])\\  \cong \inli_{i\in I} \md(F_i)(\oo_{F_i},j_i^!x^!y_* \lan r\ra[p+q]) )\cong \inli_{i\in I} \md(F_i)(\oo_{F_i},j_i^*x^!y_*   
\oo_Y\lan r-c_F\ra [p+q])\\
\cong \md(F)(\oo_{F},j_F^*x^!y_*(\oo_Y)\lan r-c_F\ra[p+q]) ;\end{gathered} $$
the last isomorphism in this chain is given by the continuity property.
\end{proof}

Now we introduce two more restrictions on $\md$. 

\begin{defi}\label{dcomp}

1. We will  
say  that $\md$ is {\it  homotopically compatible} (cf. 
 Proposition 3.2.13 of \cite{bondegl}) if
 for any $B$-field $F$ and $ r,u\in \z$ we have $\oo_F \perp \oo_F \lan r \ra [u]$    whenever $r+u>0$. 

2. We will 
say  that $\md$ is  {\it Chow-compatible} if $\oo_F \perp \oo_F\lan r \ra [u]$ whenever $F,r,u$ are as above and $u>0$.  
\end{defi}

Note that the groups $\md(F)(\oo_S,  \oo_S\lan r\ra[u])$ for $S\in G$ are certain "$\md$-versions" of motivic (co)homology groups of $S$.

\begin{theo}\label{tapplcomp}

Let $X,S,Y,H,r$ be as in Proposition \ref{pcss}(2). 

I. Assume that $\md$ is   Chow-compatible. Then the following statements are valid.

1.    $\oo_S \perp \oo_S \lan r \ra [u]$ whenever $u>0$ and $S$ is  regular. 

2. 
Assume that $Y$ is regular. 
Then $H_{BM,S}^u(X)=0$ if $u<0$.

II. Assume that $\md$ is homotopy compatible and 
$X$ is connected. Then the following assertions are valid.

1. $\oo_S \perp \oo_S \lan r \ra [u]$ whenever $u>t-r$ and $S$ is a regular $B$-scheme of Krull dimension at most $t$.

2. For the spectral sequence $T$ given by  Proposition \ref{pcss}(2) (i.e., we consider the corresponding $H$) we have $E_1^{pq}=\ns$ whenever $q>\de(Y)- r-\delta(X)$.


3. Any $\md(S)$-morphism from $\mgbms(X)$ into $y_*\q_Y\lan r\ra$ can be factored through $\mgbms(Z)$ where $Z$ is a closed subscheme of $X$ such that $\delta(Z)\le \de(Y)-r$.


\end{theo}
\begin{proof}
I. 1. 
We apply Proposition \ref{pcss}(2)  in the case $X=Y=S$ (one can apply either its "main version" or the "$\delta'$-version"  here).
We have $E_1^{pq}\cong \bigoplus_{F\in \xx^p}   \md(F)(\oo_{F},j_F^*(\oo_S)\lan r-c_F \ra[p+q])= \bigoplus_{F\in \xx^p}   
\md(F)(\oo_{F},\oo_F\lan r-c_F \ra[p+q])$; thus $E_1^{pq}$ vanishes whenever $p+q>0$ and so $H_{BM,S}^{p+q}(S)=\md(F)(\oo_F , \oo_F \lan r \ra[p+q])=0$ in this case.

2.  
We apply Proposition \ref{pcss} once again.
It certainly suffices to verify for the corresponding spectral sequence $T$ that $E_1^{pq}=\ns$ whenever $p+q>0$.
So we fix $F\in \xx^p$  and prove for $e=r-c_F$
 that 
$H_{BM,S}^{p+q}(F)\cong \md(F)(\oo_{F},j_F^*x^!y_*(\oo_Y)\lan e \ra[u])=\ns$ if $u>0$. 

Now we make certain reduction steps. 
Firstly, $X$ may be replaced by any its subscheme containing $F$; 
so we may assume that $X$ is  
(regular and) quasi-projective over $S$ and  
$F$ is its generic point. 

Consider a factorization of $x$ as $X\stackrel{f}{\to}S'\stackrel{h}{\to}S$ where
 $h$ is smooth of dimension $d$,
 $f$ is an 
embedding, $S'$ is connected, and consider the corresponding diagram
$$ \begin{CD} 
 Z@>{f_Y}>> Y'@>{h_Y}>> Y\\
 @VV{z_X}V
@VV{y'}V@VV{y}V \\
    X@>{f}>> S' @>{h}>> S
 \end{CD}$$
 (the upper row is the base change of the lower one to $Y$).
 Then we have $x^!y_*(\oo_Y)\lan e\ra= f^!h^!y_*\q_Y\lan e\ra\cong f^! y'_*h_Y^!\q_Y\lan  e\ra\cong f^! y'_*\q_{Y'}\lan e+d\ra$. Hence below we may assume  $x$ is an embedding (since we can replace $S$ by $S'$ and $e$ by $e+d$ in our calculation). 
Besides, the isomorphism $x^!y_*\cong z_{X*}z_Y^!$ for $z_Y=h_Y\circ f_Y$ yields that the group in question is zero if $Y$ lies over 
 $S\setminus X$ (considered as a set).

Next,  Remark \ref{rpur} yields that  it suffices to verify the statement for $Y$ replaced by the components of some  its regular   connected stratification.
Now, we can choose a stratification  
 of this sort such that each $Y_l^\al$  lies either over $X$ or over $S\setminus X$. 
  Therefore it suffices to verify our assertion in the case where $y$ factors through $x$. Moreover, since $x^!x_*$ is the identity functor on $\md(X)$ (in this case), 
	we may also assume that $X=S$ (and $S$ is regular and connected). 
	
	Consider the following Cartesian square: $$ \begin{CD} 
 Y_0@>{j_{0Y}}>>  Y\\
 @VV{y_{0}}V
@VV{y}V \\
    X_0@>{j_0}>> S
 \end{CD}$$
  We have $j_0^*y_*(\oo_Y) \cong
   y_{0*}j_{0Y}^*(\oo_Y)$ (see  Proposition
4.3.14 of \cite{cd}); hence this object is isomorphic to  
  $ y_{0*}\oo_{Y_0}$.
	Hence the adjunction 
  $  y_{0}^*\dashv y_{0*}$ 
	yields that the group in question is isomorphic to $\md(Y_0)(\oo_{Y_0},\oo_{Y_0}\lan e \ra[u])$ (for the "new" value of $e$ that differs from the "old" one by $d$). Now, $Y_0$ is regular (being a pro-open pro-subscheme of $Y$ that is regular).  Thus it remains to apply assertion I.1 to $Y_0$.

Alternatively, this assertion may be reduced from the previous one using the (somewhat similar) arguments of \S1.3 of \cite{bonivan}.

II.1. 
One can easily prove this assertion  by an argument that is  rather similar to that of assertion I.1; 
however, one should use the "$\delta'$-version" of this argument here. 

2. We argue similarly to the proof of assertion I.2. So we 
fix $F\in \xx$ of $\de$-codimension $p$; we should check the vanishing of 
$$H_{BM,S}^{p+q}(F)\cong  \md(F)(\oo_{F},j_F^*x^!y_*(\oo_Y)\lan r-c_F \ra[p+q])$$ if $q>\de(Y)- r-\delta(X)$. 

Now we make the same reduction steps as above; the main difference is that we should keep track of the data that changes during our reduction steps. 
When we replace $X$ by its connected subscheme $X'$ whose generic point is $F$ then we have $\de(X')=\de(F)=\de(X)-p$, whereas "the new" value $q'$ of $q$ is greater by  $p$ than the "old $q$";   $c_F$ and $p$ become $0$. Since $c_F<p$ (see Proposition \ref{pdimf}(\ref{pdf-ast})), this change is fine for our vanishing assertion, i.e.,  the "new" vanishing statement implies the "old" one. 

In the next reduction step we replace $Y$ by $Y'$ and replace $r$ by $r+d$. Now, $\delta(Y')\le \delta(Y)+d$ (we have an equality if $h_Y$ is dominant; see 
 Proposition \ref{pdimf}(\ref{pdf-3})). Once again, it follows that   if we will prove the vanishing in question  for the "new" $(Y,r)$ then we would also obtain our vanishing statement for the "old pair".

In the next  step we replace $Y$ by its subscheme (lying over $X\subset S$). Since this does not increase the value of $\de(Y)$ (and does not change $q,r,$ and $\de(X)$), the "new" vanishing statement implies the "old" one once again.

Now, the last step expresses the group in question as $\md(Y_0)(\oo_{Y_0},\oo_{Y_0}\lan r \ra[q])$ (recall that $p=c_F=0$). Now,  $\de(F)=\de(X_0)$  and $\de(Y_0)=\de(Y)$ (see Proposition \ref{pdimf}(\ref{pdf-3})); since $Y_0$ if a finite type  $F$-scheme it follows that its Krull dimension is at most $\de(Y)-\de(X)$. Thus it remains to apply assertion II.1 to $Y_0$.

3. A morphism $m$ of this sort is exactly an element of $H^0_{BM,S}(X)$ for $H$ as in Proposition \ref{pcss}(2). By Proposition \ref{pbm}(1) we reduce the assertion to the fact that $m$ is supported in $\delta$-codimension at least $\de(X)+r-\de(Y)$ (i.e., at some closed $Z\subset X$ such that $\de(Z)\le \de(Y)-r$). The latter statement is immediate from the previous assertion along with the description of the filtration corresponding to $T$ (see    Proposition \ref{pcss}(1)).

\end{proof}

\begin{rema}\label{rgenfact} Part II.3 of  our theorem is essentially (see Remark \ref{rexamples}(1) below) a vast generalization of  Proposition 2.2.6(2) of \cite{bscwh}. The arguments used in our proof are  quite distinct from that used for the proof of loc. cit. (and do not depend on ibid.).
\end{rema}

\subsection{The Chow weight structures on our (sub)categories and on their intersections  }\label{swchow}
Now we describe some more additional assumptions on $\md$. 

 We fix a set of primes $\sss$ 
and assume that all the elements of $\p\setminus \sss$ are invertible on $\sz$ (so, the characteristics of all the residue fields of $\sz$ as well as of  arbitrary $B$-schemes belong to $\sss\cup\ns$).\footnote{This is certainly a vacuous restriction in the case $\sss=\p$.} Furthermore, for   $\lam=\z[\sss\ob]$ we   suppose (in this section) that all the values of $\md(-)$  are $\lam$-linear (triangulated) categories.
We also assume that $\md(-)$ satisfies the following {\it splitting property}: if $g:F'\to F$ is a finite 
morphism of $B$-fields of  degree $d$ 
then 
the  unit morphism $u:\oo_F\to g_*g^*(\oo_{F}) \cong g_*(\oo_{F'})$ "splits up to a power of $d$"
i.e. there exists $q\in \md(F)(g_*g^*(\oo_{F}) , \oo_F)$ such that $q\circ u=d^i\id_{\oo_F}$ (for some $i\ge 0$). 

Following \cite{bonegk} (cf. also \S6.2 of \cite{cdet}), we recall that the aforementioned assumptions on $\md$ imply the following important statements.

\begin{pr}\label{pgabber}
Let $g\colon Y\to X$ be a finite type  (separated)  $\gu$-morphism. 
Then the following statements are valid.

1. 
The functors $g^*,g_*,g^!$, and $g_!$ respect the compactness of objects

2. 
Let $\delta(Y)=d$, $n\in \z$, $-d-1\le n<d$. 
 Then there exist finite type morphisms $u_{j}\colon U_j\to X$ for $-d\le j\le d$  such that $U_j^m$ are regular, $\delta(U_j)\le d-|j|$, and   
  $f_*(\oo_Y)$ belongs to the $\md(X)$-envelope of 
	$\{\mgbmx(U_j)\lan -\max(j,0)\ra [j]:\ j\le n\}\cup \{u_{j*}(\oo_{U_j})\lan  -\max(j,0)\ra [j]:\ j> n\}$.
		
\end{pr}
\begin{proof}
1. We reduce this statement to assertion 2 using (in particular) certain arguments of Cisinski and D\'eglise.

Firstly, since $\md(S)$ is compactly generated by twisted Borel-Moore objects of finite type $S$-schemes for any $S\in G$, its subcategory of compact objects
is densely generated by these objects (see  \S\ref{snotata}). Since $g_!$ sends Borel-Moore objects into  Borel-Moore ones, it does respect the compactness of objects. Next,  
Proposition \ref{pbm}(2) yields that $g^*$  respects the compactness of objects also.

Now, the arguments used in the proof of Theorem 6.4 of \cite{cdint} reduce the remaining parts of our assertion to the fact that $g_*(\oo_Y)$ is compact.
Hence it suffices to apply assertion 2 in the case $n=d$. 


2. 
In the case where $\de$ is the function introduced in  \cite[\S3.1]{bonegk} this statement is immediate  from Theorem 3.4.2 of ibid. Moreover, the arguments of ibid. can be easily combined with our somewhat more general description of $\de$.

\end{proof}

Now we prove the main result of this section.

\begin{theo}\label{tchowcap}
I. Assume that $\md$ is Chow-compatible; let $\du\subset \md(\sz)$ be the category defined in \S\ref{srmotfilt}.
Then the method described in Proposition \ref{pwexbb} yields the following weight structures (here we start with the subcategory of compact objects using this proposition and then extend this weight structure to the "big" categories using Proposition \ref{pwexb}).

1. For $\cu=\du\lan i \ra$  (see \S\ref{srmotfilt}; $i\ge 0$) 
 one can take $\cu_-=\{\mgbmsz(Y)\lan \delta(Y)+i\ra[-s]\}$  and $\cu_+=\{y_*(\oo_Y) \lan \delta(Y)+i\ra[s] \}$  for $y:Y\to B$ running through all  separated finite type  
morphisms with regular domain 
and $s\ge 0$ (note that all these objects are compact by Proposition \ref{pgabber}(1)).
 
2. $\cu=\de_{\le m}\du$ one can take $\cu_-=\{\mgbmsz(Y)\lan \delta(Y)\ra[-s]\}$ and $\cu_+=\{y_*(\oo_Y) \lan \delta(Y)\ra[s] \}$   for $y:Y\to B$ running through   separated finite type  
morphisms  with regular domain and $\de(Y)\le m$, 
and $s\ge 0$.

We will call these weight structures {\it Chow} ones.

II. Assume moreover that $\md$ is homotopy compatible. 
Then for any $i,m\ge 0$ any $w_{\du}$-bounded below element of   $ \obj \delta_{\le m} \du \cap \du\lan i\ra $ belongs to $\obj \de_{\le m-i} \du\lan i \ra$.

\end{theo}
\begin{proof}
I. The orthogonality assumptions we need (certainly, it suffices to check the one for assertion I.1) are  given by Theorem \ref{tapplcomp}(I.2). 
Now we take $\hu$ consisting of Borel-Moore motives of regular finite type $\sz$-schemes; we twist them by $\lan \delta(Y)+i\ra$ for the proof of assertion I.1, whereas for assertion I.2 we take $\mgbmsz(Y)\lan \delta(Y)\ra$ for $\de(Y)\le m$. These choices of $\hu$ densely generate the corresponding $\cu^c$ according to Remark \ref{rpurbm}.
 
The existence of the corresponding pre-weight decompositions along with the boundedness condition is given by  Proposition \ref{pgabber}(2).

II. We apply Proposition \ref{pintersw} along with Remark \ref{rintersw}(2) for $\cu_1= \delta_{\le m} \du$, $\cu_2=\du\lan i\ra$,  and $\cu_3= \de_{\le m-i}\lan i \ra \du$.
So, 
we should check that 
$\pi(\cu_{1,w^c\le 0})\perp \pi(\cu_{2, w^c\ge 0})$.
It certainly suffices 
 to verify this statement for the dense generators of $\cu_{1,w^c\le 0}$ and $\cu_{2, w^c\ge 0}$, i.e., for
$\mgbmsz(Y_1)\lan \delta(Y_1)\ra[-s_1]$ and $y_{2*}\oo_{Y_2} \lan \delta(Y_2)\ra[s_2] $, where $y_j:Y_i\to B$ ($i=1,2$) are  separated finite type  
morphisms, 
$s_1,s_2\ge 0$, and $\de(Y_2)\le m$. Now, combining  Theorem \ref{tapplcomp}(II.3) with Proposition \ref{ploc}(I.3) we obtain this statement for $s_1=s_2=0$,
 whereas for $c_1+c_2>0$ this orthogonality is a consequence of the orthogonality property of the Chow weight structure for $\du$.

\end{proof}

\begin{rema}\label{rexamples}
1. The main subjects studied in \cite{bondegl} were  motivic categories satisfying all our assertions except (possibly) the Chow-compatibility one.\footnote{Actually, some of the categories considered in ibid. were not oriented. The orientability condition does not seem really necessary for the arguments of this section at least provided that the splitting condition 
is available
 (though it allows us not to mention the corresponding Thom objects in our formulae); yet the author chose to impose it since he does not know of any Chow-compatible examples that are not orientable; cf. Remark \ref{rintersk}(3).}  Moreover, Proposition \ref{pgabber} makes the assumption (Resol) of \cite[\S2.4.1]{bondegl} superfluous (at least, for our purposes).
Thus we obtain three examples of our statement; 
cf. Example 1.3.1 of ibid. The first two of them were studied in \cite{brelmot} (see also \cite{hebpo}) and in \cite{bondegl}, respectively. 

(i) For any Noetherian separated excellent scheme $B$ of finite Krull dimension one can consider the categories of Beilinson motives over 
 $B$-schemes; recall that  Beilinson motives are certain Voevodsky motives with rational coefficients that were considered in detail in \cite{cd}. 

(ii) For any (Noetherian separated excellent finite dimensional) scheme $B$ of characteristic $p$ and any $\zop$-algebra $R$ (recall that we set $\zop=\z$ if $p=0$) one can consider $R$-linear $cdh$-motives  over  
 $\sz$-schemes (this is another version of Voevodsky motives).

(iii) For any $B$ as in (ii) and any set of primes $S$ containing $p$ (and $\lam =\z[S\ob]$) one can consider the $\lam$-linear version of  the categories $\shmgl(-)$, where the latter are the (stable) homotopy categories of the categories $\mglmod$ of motivic spectra (i.e., we consider "Quillen models" for $\sht(-)$) 
 endowed with the structure of (strict left) modules over the Voevodsky's 
 spectra $\mgl_-$.

Note here that the Chow-compatibility property is a statement over 
fields; so it is well-known for examples (i) and (ii), and follows from Theorem 8.5 of \cite{hoycobord} 
for example (iii).

2. Yet the condition (Resol) (that ensures the "abundance" of regular proper $S$-schemes for any $S\in G$) certainly simplifies the consideration of motivic categories. So, if we assume that $B$ is of finite type over a scheme of dimension at most $3$ (in the examples described above)
 then Theorem \ref{tchowcap} can be proved similarly to Theorem \ref{tintersk}  (and so, using Proposition \ref{pwexb}  instead of \ref{pwexbb}); cf. the discussion of Chow motives over a base in \S2.3 of \cite{bonivan}. 

3.  Certainly, taking $\cu_-=\{\mgbmsz(Y)\lan i\ra[-s]\}$  and $\cu_+=\{y_*(\oo_Y) \lan i\ra[s] \}$  for $y:Y\to B$ running through all  separated finite type  
morphisms with regular domain and $i\in \z$  one obtains (similarly to part I of the theorem) a certain (Chow) weight structure on the whole $\md(\sz)$; the embedding $\du\to \md(\sz)$ is weight-exact (see 
Proposition \ref{pwexb}(2)).  

It is easily seen that this weight structure on $\md(\sz)$ coincides with the corresponding Chow weight structure  that can be constructed for any $\md$ satisfying our assumptions using the methods of (\S2 of) ibid. 
In particular,  an object of $\du$ is $w_{\du}$-bounded below whenever it is $\wchow$-bounded below in $\md(\sz)$ (see Proposition \ref{pbw}(\ref{iwemb})).

Moreover, one certainly 
 has no need to treat $\de$ when constructing $\wchow$  on the whole $\md(\sz)$. 

4. Since the objects $\mgbmsz(Y)\lan i\ra $  (for $y:Y\to B$ running through all  separated finite type  
morphisms with regular domain and $i\in \z$) densely generate $\md(\sz)$, any compact object in this category is $\wchow$-bounded below. 

5. Moreover, for all of the motivic categories described above 
one can prove a certain relative version of Proposition \ref{pintersthom}. Note firstly that the $\delta$-homotopy $t$-structure was defined 
 in \cite{bondegl} using a natural generalization of the description of $\thom$ given in Remark \ref{rdeg}(2). Thus the proof of our Proposition \ref{pintersthom} can be easily generalized to this context provided that a certain relative version of Lemma \ref{layoub} is available. It appears to be no problem to deduce some statement of this sort from the case where the base is a field using Proposition \ref{pcss}. Moreover, if our motivic categories are stable under purely inseparable extensions of fields (which is certainly the case for our examples) then  can assume this base field to be perfect; thus one can apply Lemma \ref{layoub} itself for examples (i) and (ii), whereas the arguments (from the proof of Proposition 4.25) of \cite{ayconj} appear to work without any problem for example (iii). 

Furthermore, this argument can certainly be "axiomatized". 

\end{rema}

\end{document}